\documentclass{article}

\usepackage{PRIMEarxiv}

\usepackage[utf8]{inputenc} %
\usepackage[T1]{fontenc}    %
\usepackage{hyperref}       %
\usepackage{url}            %
\usepackage{natbib}
\usepackage{booktabs}       %
\usepackage{amsfonts}       %
\usepackage{nicefrac}       %
\usepackage{microtype}      %
\usepackage{lipsum}
\usepackage{fancyhdr}       %
\usepackage{graphicx}       %
\graphicspath{{media/}}     %
\usepackage{nicefrac}
\usepackage{mathtools}
\usepackage{mathrsfs}
\usepackage{amsmath, amssymb, amsopn, amsthm} 
\usepackage[dvipsnames]{xcolor}
\usepackage{svg}
\usepackage{epstopdf}
\definecolor{mpl_green_C2}{HTML}{2ca02c} %
\usepackage[capitalize,noabbrev,nameinlink]{cleveref}
\usepackage{graphicx} %
\usepackage{caption}  %
\usepackage{subcaption} %
\usepackage{wrapfig}
\usepackage{enumitem}

\usepackage[toc,page,header]{appendix}
\usepackage{titletoc}

\numberwithin{equation}{section}
\newtheoremstyle{exampstyle}
{4pt} %
{4pt} %
{\itshape} %
{} %
{\bfseries} %
{.} %
{.5em} %
{} %
\theoremstyle{exampstyle}

\newcommand{\be}{\begin{equation}}
\newcommand{\ee}{\end{equation}}
\newcommand{\bes}{\begin{equation*}}
\newcommand{\ees}{\end{equation*}}

\newcommand{\R}{\mathbb{R}}

\newcommand{\mc}[1]{\mathcal{#1}}

\newtheorem*{theorem*}{Theorem}
\newtheorem{theorem}{Theorem}[section]
\newtheorem{proposition}[theorem]{Proposition}
\newtheorem{lemma}[theorem]{Lemma}
\newtheorem{corollary}[theorem]{Corollary}
\newtheorem*{corollary*}{Corollary}
\newtheorem*{lemma*}{Lemma}
\theoremstyle{definition}
\newtheorem{definition}[theorem]{Definition}
\newtheorem{assumption}[theorem]{Assumption}
\newtheorem{nb}[theorem]{Note}
\newtheorem{example}[theorem]{Example}

\theoremstyle{remark}
\newtheorem{remark}[theorem]{Remark}
\theoremstyle{definition}

\DeclareMathOperator{\Orb}{\operatorname{Traj}}

\newcommand{\TpM}{T_p\mathcal{M}}
\newcommand{\Tp}[1]{T_p{#1}}

\newcommand{\Rl}{\mathbb{R}}

\DeclareMathOperator{\im}{Im}

\usepackage{authblk}

\makeatletter
\renewcommand\AB@affilsepx{, \protect\Affilfont}
\makeatother

\newtheorem*{definition*}{Definition}
\newtheorem*{nb*}{Note}
\crefname{lemma}{Lemma}{Lemmas}
\crefname{proposition}{Proposition}{Propositions}

\AddToHook{env/proposition/begin}{\crefalias{theorem}{proposition}}
\AddToHook{env/lemma/begin}{\crefalias{theorem}{lemma}}
\AddToHook{env/corollary/begin}{\crefalias{theorem}{corollary}}
\AddToHook{env/definition/begin}{\crefalias{theorem}{definition}}
\AddToHook{env/assumption/begin}{\crefalias{theorem}{assumption}}
\AddToHook{env/nb/begin}{\crefalias{theorem}{nb}}
\AddToHook{env/example/begin}{\crefalias{theorem}{example}}
\AddToHook{env/prob/begin}{\crefalias{theorem}{prob}}
\AddToHook{env/remark/begin}{\crefalias{theorem}{remark}}

\title{

When is a System Discoverable from Data? \\
Discovery Requires Chaos

}

\author{Zakhar Shumaylov\textsuperscript{1,}\thanks{Equal contribution.}\;,
  Peter Zaika\textsuperscript{1,}\protect\footnotemark[1]\;,
  Philipp Scholl\textsuperscript{3}, \authorcr
  Gitta Kutyniok\textsuperscript{3},
  Lior Horesh\textsuperscript{2},
  Carola-Bibiane Sch\"onlieb\textsuperscript{1}}
\affil{\textsuperscript{1}University of Cambridge \quad
  \textsuperscript{2}IBM Research  \\
  \textsuperscript{3}LMU Munich and Munich Center for Machine Learning (MCML)}
\begin{document}
\maketitle

\begin{abstract}
The deep learning revolution has enabled major AI advances in physical sciences, particularly in discovering dynamical systems from observational data. Yet the reliability of learned surrogates and symbolically discovered models is often undermined by non-uniqueness: models may fit observed trajectories perfectly, but lack genuine predictive power. We show that \emph{chaos}, despite its association with unpredictability, is often what makes governing equations discoverable. We prove that \emph{the classical Lorenz system is analytically discoverable} and verify experimentally that for many systems out-of-domain generalization emerges precisely at chaos onset. Conservative systems, by contrast, are inherently non-discoverable from data alone. The prevalence of chaotic systems in benchmarks may have inadvertently obscured this limitation. These findings help explain why AI excels at weather forecasting, yet struggles in many engineering applications: the systems most amenable to discovery (chaotic ones) are often least amenable to control.
\end{abstract}

\section{Introduction}
\noindent The rise of deep learning over the last decade has led to a paradigm shift in modeling approaches, resulting in a move away from first principles deductive models towards inductive models utilizing ever-increasing pools of data. Artificial Intelligence (AI)-driven scientific discovery has spurred remarkable progress across chemistry \citep{mater2019deep}, physics \citep{ degrave2022magnetic}, biology \citep{jumper2021highly}, with a growing suite of methods promising to automate the process of uncovering governing physical laws and models directly from data \citep{brunton2024promising,vinuesa2022enhancing}. This ambition to automate derivation of models, whether as learned surrogates \citep{azizzadenesheli2024neural} or through symbolic discovery techniques \citep{brunton2016discovering,cranmer2019learning,cory2024evolving}, represents a potential revolution in scientific methodology.

However, the practical feasibility \citep{greivy} and reliability \citep{de2023physics, grossmann2024can} of these models has often fallen short of expectations, despite reported promising results \citep{Nature}. Traditionally, the cause of such inconsistencies has been attributed to a ``reproducibility crisis'' \citep{baker2016reproducibility}, especially fueled by the rapid progress of deep learning \citep{kapoor2023leakage}. 

\begin{figure}[t]
    \centering
    \includegraphics[width=\textwidth]{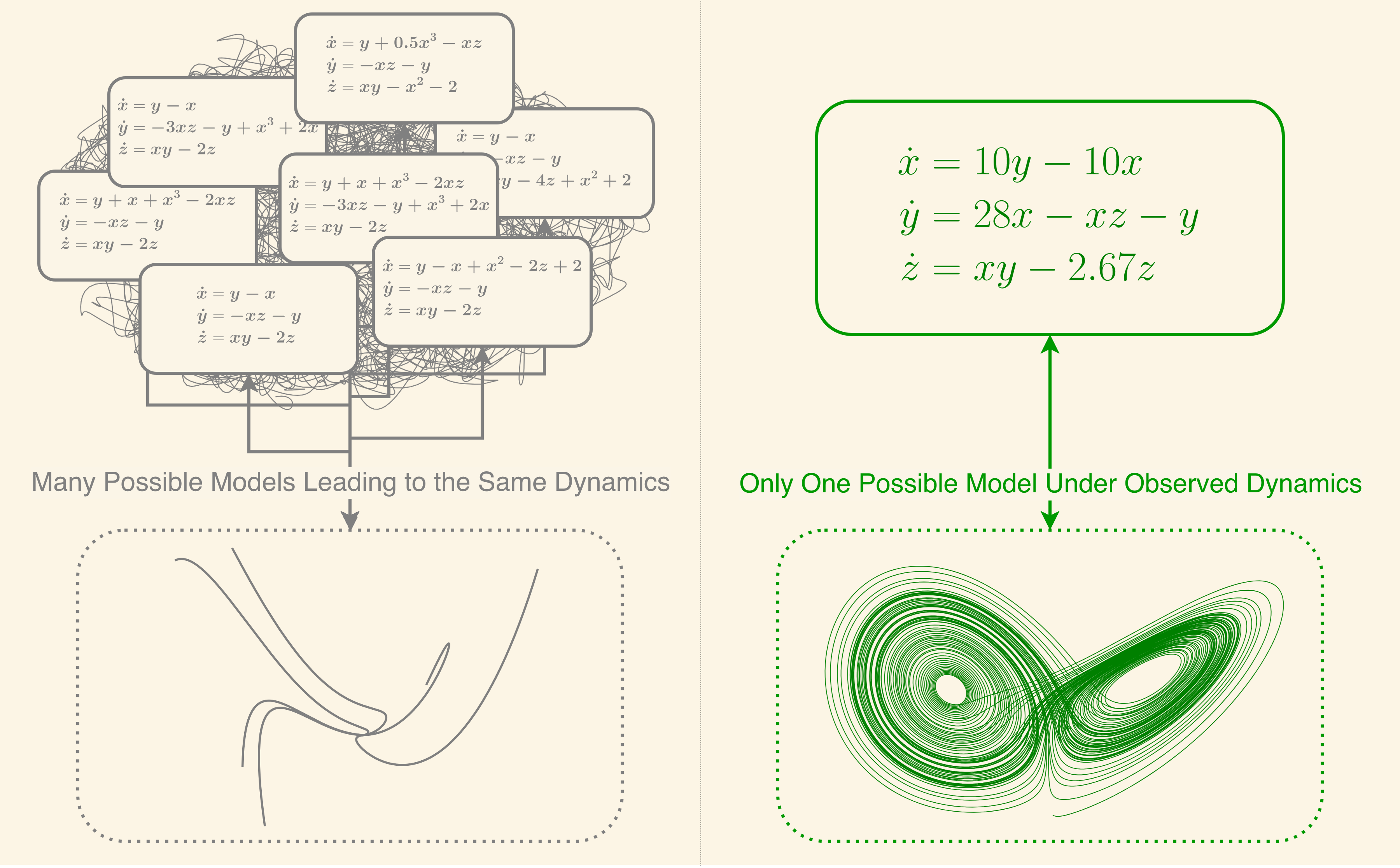}
    \caption{\textbf{Model ambiguity in non-chaotic systems versus uniqueness in chaotic systems.} \textbf{(Left)} For non-chaotic systems like the integrable Lorenz model, different mathematical models can generate identical dynamics, precluding unique discovery (\Cref{thm:first_integral_main}). \textbf{(Right)} Conversely, the dynamics of a chaotic system like the Lorenz attractor uniquely specify the underlying model, making it discoverable (\Cref{prop:geom_anal}).}
    \label{fig:Lorenz_discoverability}
\end{figure}

We argue that the real limitation of dynamics discovery is not methodological. It stems from a foundational issue tied to the nature of the systems being studied. The enterprise of data-driven discovery rests on a critical, yet often implicit, assumption: that the observational data is sufficient to uniquely specify the underlying system. When this assumption fails, e.g. illustrated in \Cref{fig:Lorenz_discoverability}, the discovery process is compromised. The resulting models may fit the available data perfectly but lack genuine predictive power or physical reality \citep{vafa2025foundationmodelfoundusing}, undermining the reliability required in scientific applications. 
While relatively new to scientific machine learning, questions of uniqueness and stability of solutions under noisy observations are central to the field of inverse problems \citep{arridge2019solving,schonlieb2025data}, which classifies problems as well or ill-posed based on the criteria established by Hadamard \citep{hadamard1902problemes}. A problem is \emph{well-posed} if a solution \emph{exists}, is \emph{unique}, and \emph{depends continuously on observed data}. 

The uniqueness question has been studied in the literature under the term ``identifiability'' \citep{bellman1970structural,distefano1980parameter,nguyen1982review}, largely confined to restrictive settings, such as linear systems, which are more akin to parameter identification than true model discovery. We extend this inquiry to the much broader and more practical function classes relevant to modern deep learning: continuous ($C^0$) and real analytic ($C^\omega$) spaces, a concept we term \emph{discoverability}.

Within this framework, we show that \emph{chaos} is one of the key conditions for ensuring well-posedness of dynamical system discovery. While chaos is not technically necessary in all function classes, it does provide a robust and sufficient mechanism for discovery: chaotic dynamics generate trajectories that explore and mix through a sufficiently rich portion of state space, yielding uniqueness and, under quantitative mixing assumptions, stability. Consequently, many discovery methods appear robust because they are tested on systems where the assumption of uniqueness naturally holds, yet they fail when applied to non-chaotic dynamics, where it does not. With benchmark datasets built predominantly from chaotic systems, this problem has gone largely unnoticed \citep{wyder2025commontaskframeworkcritical,kaptanoglu2023benchmarking,chesebro2025scientificmachinelearningchaotic,table_1}. Importantly, our notion of discoverability concerns what can be inferred from a single, or finitely many, trajectories, rather than what becomes possible as the number of initial conditions increases and progressively covers the state space. The central question is what a trajectory reveals about the governing system, rather than the sampling complexity required to identify the system from many trajectories.

In this article, we completely characterize discoverability: when a system is discoverable, when it is not, and show how conservation laws can recover well-posedness. Discoverability hinges on trajectories that persistently explore and mix through a sufficiently large amount of space -- precisely what chaotic dynamics guarantee. As a consequence, we prove that the system generating the Lorenz attractor is discoverable (\Cref{fig:Lorenz_discoverability}), which may help explain its reliability and persistence as a benchmark: the system is not only identifiable, but is  stably recoverable under finite-time noisy observation. We confirm experimentally that out-of-domain generalization is impossible below the onset of chaos, and emerges precisely as the system becomes chaotic (\Cref{fig:lorenz_ood}).

\begin{figure}[t]
    \centering
        \centering
        \includegraphics[width=.6\textwidth]{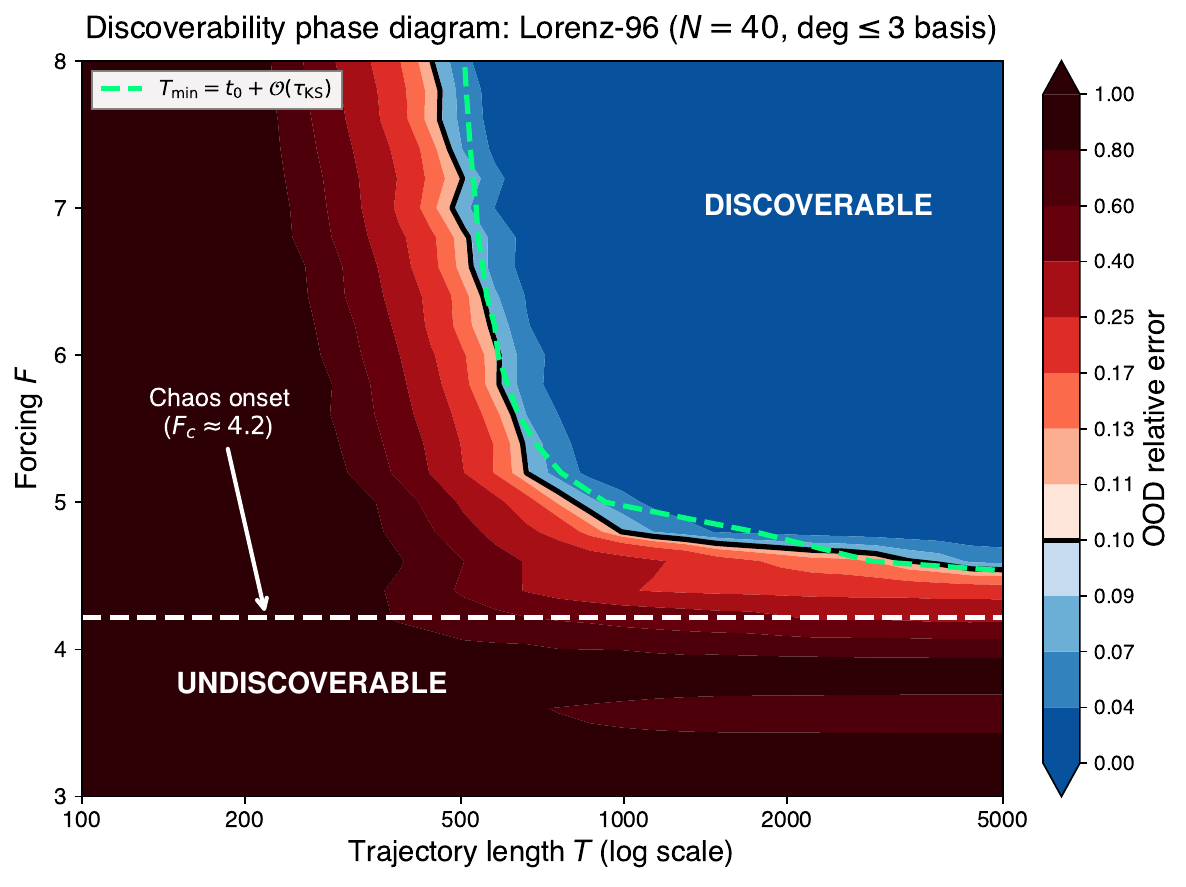}
        \label{fig:lorenz_ood_error}

    \caption{\textbf{Discoverability maps for the Lorenz-96 system}. Out-of-distribution (OOD) prediction error for various forcing parameters under different trajectory lengths, capturing the phase transition across the chaos boundary, plotted in finite time using the Kolmogorov-Sinai entropy. Details are provided in \Cref{sec:comp_details}, with similar plots for R\"ossler and Henon-Heiles systems. }
    \label{fig:lorenz_ood}
\end{figure}

\subsection{Historical Context}
Historically, the challenge of deducing governing laws directly from observational data is centuries-old, long predating modern computational methods. It began with Kepler's derivation of planetary kinematics from Brahe’s astronomical datasets \citep{1609kepler}, igniting the scientific revolution. This tradition was formalized in the mid-20th century as system identification: driven by demands of aerospace engineering for stable feedback loops, \r{A}str\"om, Ho and Kalman shifted the focus from physical modeling to generalized ``black-box'' techniques for automatic extraction and control of linear state-space representations from trajectory data \citep{bellman1970structural,aastrom1971system,ho1966effective}. 

The fully nonlinear setting, however, was charted by a separate mathematical lineage, tracing back to Poincar\'e's seminal work on the three-body problem \citep{poincare1890probleme}. He discovered that deterministic systems could exhibit homoclinic tangles and chaotic dynamics, establishing impossibility of obtaining closed-form solutions for many-body Hamiltonian systems and revealing how sensitive dependence on initial conditions undermines predictability. What Poincar\'e viewed as a barrier to predicting future states, paradoxically became the key to understanding observational data a century later. In the 1980s Takens delay embedding theorem formally demonstrated that the complete topology of a chaotic system can be reconstructed from a single trajectory \citep{sauer1991embedology,takens2006detecting}. Building upon this geometric foundation, the natural culmination is to move beyond merely reconstructing the topology of the attractor, to fully discovering the system that generates it. 

It is precisely this final gap that our present work resolves. We establish that the limits of equation discovery are inextricably linked to this historical lineage, proving that the very chaotic dynamics Poincar\'e identified as destroying predictability are exactly the ones required to uniquely recover a system's governing laws from data.

\begin{figure}[t]
\centering
\includegraphics[width=\textwidth]{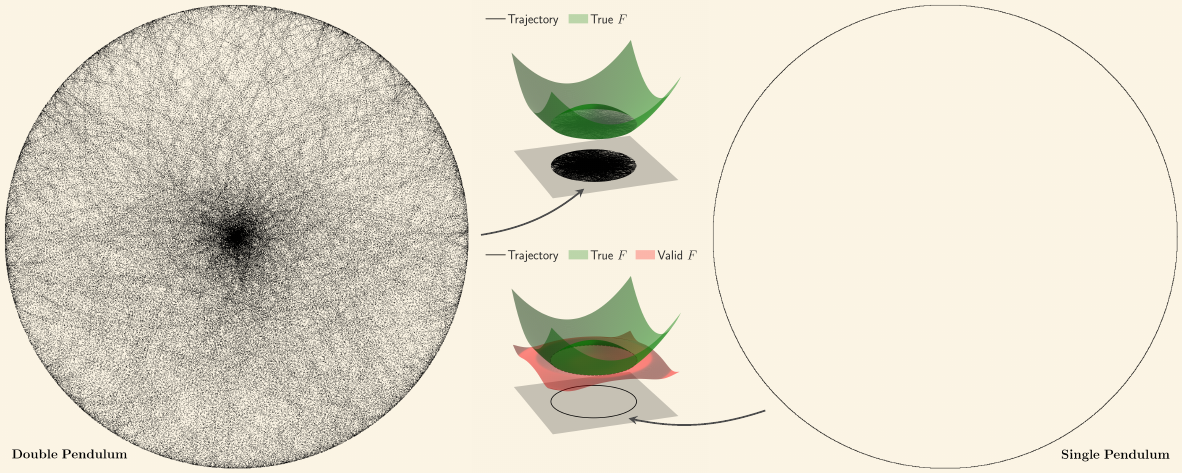}
\caption{\textbf{An illustration of how system dynamics affect model discovery from data}. The regular motion of the simple pendulum explores a limited portion of its state space, insufficient to uniquely determine its governing equations. In contrast, the chaotic double pendulum explores a much richer portion of its accessible state space, illustrating how chaotic dynamics can generate substantially more informative trajectories. Inspired by \href{https://github.com/profConradi}{Simone Conradi}.}
\label{fig:pendulums}
\end{figure}
\section{Minimal Mathematical Background}
In what follows we consider dynamical systems and their flows of the form:  
\begin{equation}\label{eq:basic}
\dot u = F(u) \quad \text{with flow} \quad \dot{\phi}(x_0, t) = F\left(\phi(x_0, t)\right),
\end{equation}
where $u:\mathbb{R}\rightarrow\mathbb{R}^d$, and the flow $\phi(x_0,t)$ is simply the map that takes in the initial condition $x_0$ and returns its evolution at time $t$. We will mostly concern ourselves with compact domains $X\subset\mathbb{R}^d$. For the sake of brevity, we assume these objects can be defined, see \Cref{sec:mathback} for the full setup. For such systems, we define discoverability as uniqueness on trajectories:
\begin{definition*} [Discoverability] \label{def:ndiscoverable-dynamical-systems} 
    We say that $F$ is discoverable from $n$ trajectories in $V$, if there exist $u_1,\dots, u_n:U\rightarrow\mathbb{R}^d$ differentiable trajectories on the open set $U\subset\mathbb{R}$, such that $F$ is the unique function in $V$ for which  $\dot u_i = F(u_i)$ holds for all $n$ trajectories.
\end{definition*}
\noindent In this article, we primarily focus on chaotic systems due to their discoverability properties, which we establish below. The study of chaotic systems spans many different fields, including physics, computer science, and mathematics, with various definitions capturing different aspects of what we intuitively regard as chaotic behavior, without an all-encompassing definition. However, almost all topological frameworks for chaos rely on the fundamental property of \emph{topological transitivity} \citep{blanchard2008topologicalchaosmean}:
\begin{definition*}[Topological Transitivity]
The dynamical system $F$ with flow $\phi^t$ is \emph{topologically transitive} on $X$ if for any non-empty, open sets $U, V \subset X$, exists some $t > 0$ such that
    $
        \phi^t(U) \cap V \neq \emptyset.
    $
\end{definition*}
\noindent Intuitively, this means that chaotic systems must be sufficiently mixing. In the following sections, the main object of study will be trajectories of the system, denoted as $\gamma(x)=\left\{\phi^t(x): t \in [0,+\infty]\right\}$.
We will call $x \in X$ a \emph{hypercyclic point} for $\phi$, if $\gamma(x)$ is dense in $X$. The main result we shall rely on is Birkhoff's theorem on topological transitivity \citep{birkhoff_surface_1922}, which we state here for the setting of flows: 
\begin{lemma*}[Birkhoff]\label{lem:birk}
     Let $X$ be a complete, separable, metric space with a continuous dynamical system $\phi$. Then $\phi$ is topologically transitive if and only if it has a hypercyclic point. 
\end{lemma*}
\noindent Most importantly, many chaotic systems of interest are topologically transitive only on a part of the domain, referred to as attractors. The usual presentation of chaos would go on to define trapping regions, invariant sets and sensitivity. However, the gap between transitivity and sensitivity is thin: transitive systems are eventually sensitive \citep{yorke_revisited}. For brevity, we define our main object of interest:
\begin{definition*}
    The $\omega$-limit set of the trajectory $\gamma$ is defined as
    \[
        \smash{\lim_{\omega} \gamma = \underset{s \in [0,\infty)}{\bigcap} \overline{\{\gamma(t) \mid t > s\}}.}
    \]
    A generalized limit cycle is a trajectory $\Omega$ such that $\lim_{\omega} \Omega = \overline{\Omega}$.
\end{definition*}
\noindent These $\omega$-limits have many nice properties, which we go over in detail in \Cref{sec:attractors}. Most importantly, if $\gamma$ is contained in the basin of attraction of some attractor, then $\lim_\omega \gamma$ will be inside of this attractor. One can view $\lim_\omega \gamma$ as the attractor of the specific trajectory $\gamma$.   

\section{Discoverability}
\begin{figure}[t]
    \centering
    \includegraphics[width=0.6\textwidth]{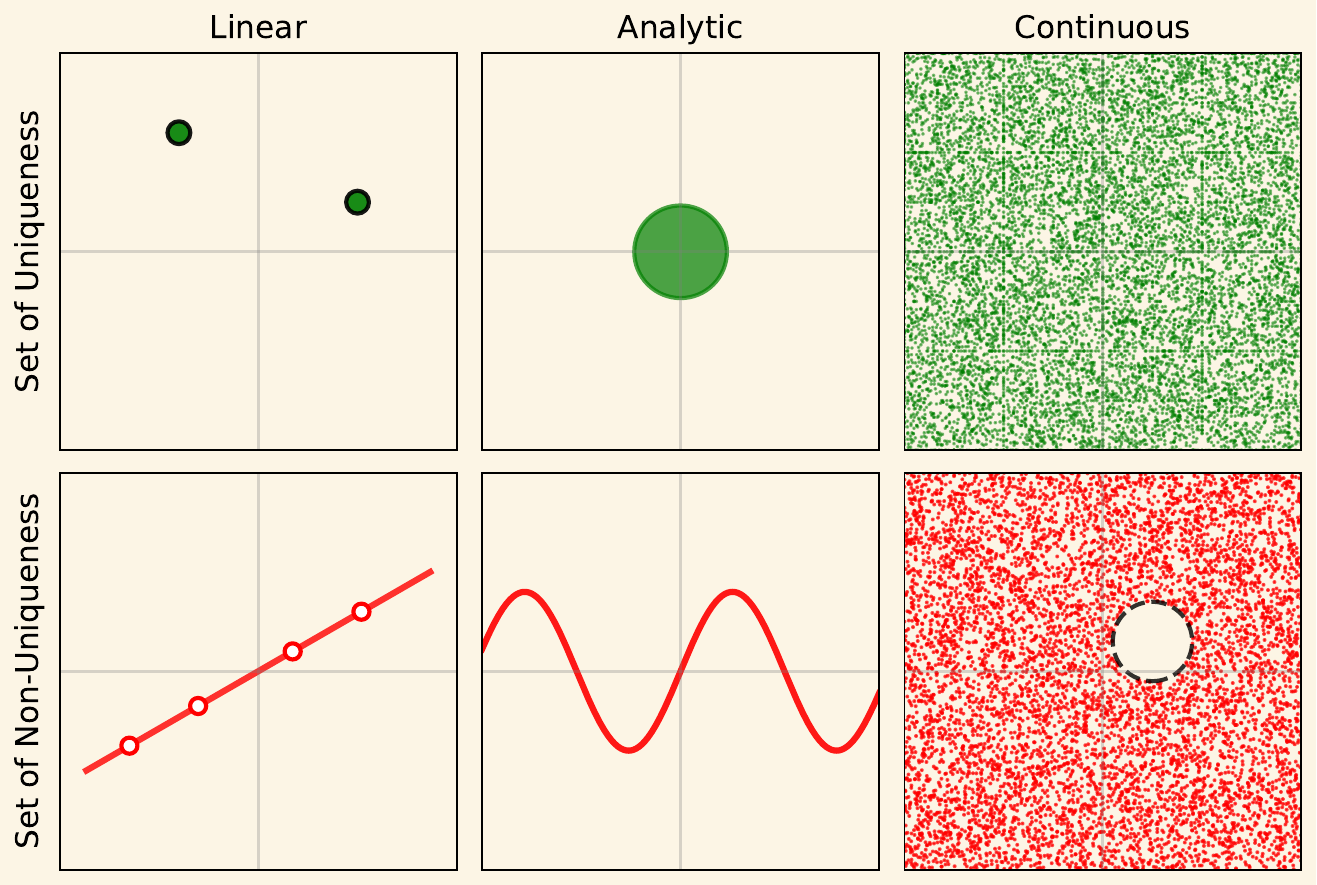}
    \caption{\textbf{Examples of \textcolor{mpl_green_C2}{sets of uniqueness} and \textcolor{red}{non-uniqueness} for various function spaces.} Columns differentiate between function spaces: Linear%
    , Real Analytic, and Continuous. Being a set of uniqueness for a function space implies a function, zero on that set, is zero everywhere.}
    \label{fig:sofu}
\end{figure}

With all the definitions now introduced, we can move onto the main topic of this study: \emph{discoverability}. The main result, allowing us to tackle the problem of discoverability is the following reduction to the study of sets of uniqueness (\Cref{fig:sofu}) based on \citet[Proposition 1]{scholl2022symbolicrecoverydifferentialequations}.
\begin{definition*}[Set of Uniqueness]\label{def:Concrete-Def-Uniqueness}
    Let $V$ be any class of functions mapping from $\mathbb{R}^d$ to $\mathbb{R}^d$. A set $A$ is called a \emph{set of uniqueness} for $V$, if for all $ G\in V$, $G|_A=0$ implies $G|_{\mathbb{R}^d}=0$.
\end{definition*}
\begin{lemma} \label{prop:general-uniqueness_main}
    Let $V$ be any class of functions mapping from $\mathbb{R}^d$ to $\mathbb{R}^d$ closed under addition and subtraction, and $F\in V$, with $\dot u = F(u)$ for $u:U\to\mathbb{R}^d$. 
    Then, $F$ is the unique function in $V$ with $\dot u = F(u)$ if and only if $u(U)$ is a set of uniqueness for $V$.
\end{lemma}
\noindent This lemma turns the hard problem of uniqueness into a simpler problem of establishing when a given trajectory covers a \emph{set of uniqueness} for a given function class. For continuous functions $C^0$, the classification is exact:
\begin{lemma}\label{prop:cont_uniq}
Set $\mathcal{D}\subset \mathbb{R}^d$ is a set of uniqueness for $C^0(X)$ if and only if $\mathcal{D}$ is dense in $X$. 
\end{lemma}
\noindent In the case of real analytic functions, the problem becomes increasingly more complicated. To the best of our knowledge there exists no geometric equivalence statement to tell us when a given set is a set of uniqueness, as some pathological examples exist. Intuitively, analytic sets of uniqueness need to be sufficiently large dimension-wise locally. Out of many existing sufficient conditions, only the following turns out to be useful for establishing discoverability of dynamical systems:
\begin{lemma}\label{lem:haus_main}
Set $A\subseteq\mathbb{R}^d$ with Hausdorff dimension $\operatorname{dim}_H A > d-1$ are analytic sets of uniqueness.
\end{lemma}
\section{Main Results}
Having reduced the problem of discoverability to that of sets of uniqueness, we now investigate whether systems trajectories can sufficiently cover such sets, leading to the main results of this article. The continuous case can be cleanly deduced from results of the previous section:
\begin{theorem}\label{thm:cont}
A system is discoverable (from a single trajectory in $X$) in the space of continuous functions $C^0(X)$ if and only if it is topologically transitive.
\end{theorem}
\noindent A similar result can be deduced in the more practical setting of multiple trajectories, provided in \Cref{pro:continuous-discoverable-multiple-trajectories}. For analytic vector fields, geometric density is no longer necessary for uniqueness, but nevertheless gives a useful sufficient condition.
\begin{proposition}\label{prop:geom_anal}
An analytic system that is chaotic on an attractor $\Lambda\subset\mathbb{R}^d$ with a Hausdorff dimension $\operatorname{dim}_H \Lambda > d-1$ is analytically discoverable from a single trajectory. If $\operatorname{dim}_H \Lambda > d - 2$ then it is discoverable from almost every trajectory approaching the attractor. 
\end{proposition}

\begin{corollary}
    System, generating the classical Lorenz attractor is analytically discoverable.
\end{corollary}
\begin{proof}
This is a direct application of proposition above, as the Hausdorff dimension of the geometric Lorenz flow is strictly larger than $2$ \citep{moreira2020hausdorff}, and is valid for the actual Lorenz system \citep{tucker1999lorenz,tucker2002rigorous} ($\dim_H$ estimated to be $2.06\pm0.01$). 
\end{proof}

\noindent \Cref{prop:geom_anal} suggests that many known chaotic systems, like the Rossler ($D \approx 2.01$), multi-scroll ($D \approx 2.13$), Chen and L\"u ($D \approx 2.18$ and $D \approx 2.06$), hyperchaotic Rossler ($D \approx 3.006$) and Rabinovich–Fabrikant ($D \approx 2.19$) systems may also be discoverable. However, establishing this rigorously is difficult, as it is necessary to both establish existence of an attractor, and exact lower bounds on its dimension. As noted by \citet{moreira2020hausdorff}, such results are scarce in the literature. \newline
\noindent This criterion is only sufficient however. In general, analytic discoverability can depend not only on the $\omega$-limit set, but also on how the trajectory approaches it.%
\begin{example}\label{ex:NonChaoticDiscoverableSystem}
Consider the system $F(u)=( -u_2 - \epsilon u_1, u_1 - \epsilon u_2)$ on $\mathbb{R}^2$, with $\epsilon > 0$.  Every nonzero trajectory of this system spirals toward the origin, so its $\omega$-limit set is the origin, which is not a set of analytic uniqueness. Nevertheless, the trajectory itself is a uniqueness set. Indeed, every line $\ell$ through the origin intersects the spiral infinitely often, with accumulation at $0$. By identity theorem, this gives $f|_\ell = 0$ for for every such line, and hence $f=0$.
\end{example}
\noindent Thus, the limiting set alone does not determine analytic discoverability; any complete criterion must retain information about the full trajectory, not only its limiting geometry. The complete characterization turns out to be algebraic, in terms of cofactors: the meromorphic function defined by $\nabla g \cdot F = K g$.
\begin{theorem}\label{thm:main_anal_main}
The system generated by $F$ is discoverable from $\gamma$ if and only if $F$ does not admit any non-zero analytic $g$ with locally bounded cofactor.
\end{theorem}

\noindent Thus, the only thing precluding discoverability are nontrivial semi-invariants, of which conservation laws (semi-invariants with $K=0$), are an important special case.
\begin{corollary}\label{thm:first_integral_main}
If an analytic vector field $F$ admits a non-constant global analytic first integral $g$, then no trajectory of $F$ is a set of uniqueness, and thus $F$ is not analytically discoverable.    
\end{corollary}
\paragraph{Structure-informed discovery:}Corollary~\ref{thm:first_integral_main} concerns discovery in the unrestricted analytic class. When a conservation law is known in advance, however, it may instead be used to restrict the admissible vector fields and thereby restore uniqueness. Results above can no longer be used, as the class itself is no longer linear. The following gives one mechanism to restore discoverability.

\begin{theorem}%
\label{thm:full_uniq_symmjac}
An analytic system $F$, admitting a conservation law $H$ with $\nabla_{1} H(\dot{u}, u) = 0$ for any $u$ solution, is discoverable in the space of vector fields with (skew-)symmetric Jacobian if for trajectory $\gamma$, there exists an analytic submanifold $\mc{S}$ and a point $p\in\mathcal{S}\subseteq\overline{\gamma}$, with $\ker{\nabla_{1}^2 H(F(p), p)} \subseteq \Tp{\mc{S}}$.
\end{theorem}
\noindent The condition is readily verified in simple examples. For example, the two dimensional harmonic oscillator is discoverable within the space of vector fields with skew-symmetric Jacobians from any single trajectory away from the origin, given the conservation law 
$H = \frac12(x\dot{x}+y\dot{y})^2$. Thus, in order to be able to discover non-chaotic systems, it becomes necessary to have a-priori knowledge.

\section{On Full Well-Posedness}
Having established the discoverability of chaotic systems, we now address the second critical component of well-posedness: stability. Specifically, we investigate the behavior of the reconstruction when we do not observe the ground-truth trajectory, but rather a noisy counterpart $y(t) = x(t) + \varepsilon(t)$. For analysis, we consider reconstructions via ordinary least squares (OLS).

For finite-dimensional function spaces, it turns out that stability and convergence under noise can be established without any explicit regularization. In such settings, the Gram matrix is strictly invertible, meaning the problem is naturally well-posed -- once again via a reduction to sets of uniqueness. Combining this with an explicit mixing rate, often present in chaotic systems, gives the following stability result:

\begin{theorem}
\label{thm:stability_intrinsic}
Assume the true dynamics $\dot{x} = F(x)$, with $F\in\mathcal{H}$ finite dimensional vector space, admit an invariant exponentially mixing measure $\mu$ supported on a set of uniqueness $\Lambda$, and observations $y(t) = x(t) + \varepsilon(t)$ have bounded noise $\|\varepsilon(t)\|_2 \leq \delta$ and $\|\dot{\varepsilon}(t)\|_2 \leq \delta_v$. Then with high probability over the initial condition $x_0\sim\mu$, for noise below a critical threshold and sufficiently long observation time, the error of OLS estimator $\hat{F}_T$ is bounded as (with problem specific constants for each term in the sum):
$$
\smash{\|\hat{F}_T - F\|_{L^2(\mu)}  \leq \mathcal{O}\left(\delta_v+\delta+\sqrt{\frac{1}{T}}\right)}
$$
Therefore, as $\delta,\delta_v\to0$ and $T\to \infty$, $\hat{F}_T\to F$ in $L^2(\mu)$, in probability.
\end{theorem}
\noindent Crucially, analogous stability and convergence results can be derived for Hilbert spaces. However, the infinite-dimensional case fundamentally requires regularization; we provide that analysis in \Cref{app:infinite_stability}.

\section{Discussion}

The central finding of this work is theoretical: within the space of continuous or analytic functions, chaos is almost necessary for provable  discoverability. Many apparent failures of data-driven methods may thus reflect the fundamental ill-posedness of the problems, rather than methodological flaws. This does not mean non-chaotic systems can never be discovered in practice; successful algorithms must instead operate in restricted (e.g., finite or sparse) search spaces, utilize significantly more trajectories and cannot be expected to generalize. 

Practically, discoverability is best understood via the informational content of trajectories: at one extreme, trajectories at stationary points provide no information for system discovery; at the other,  chaotic systems admit trajectories rich enough for exact reconstruction. This perspective is further confirmed by the empirical success of maximizing state-space coverage \citep{schaeffer2018extractingsparsehighdimensionaldynamics, Wu_2019}. 

\paragraph{Discoverability for black-box models:} For symbolic discovery, the idea behind discoverability is clear-cut -- recovering $(x^2+y^2)\cdot x\cdot y$ is demonstrably different from $x\cdot y$. For black-box models such as neural operators, by contrast, discoverability manifests as an upper bound on generalization. A model trained on a non-discoverable system can generically exhibit limited generalization, while sufficiently expressive models trained on discoverable systems \emph{must} generalize due to uniqueness. And recent evidence suggests that applying neural operators to chaotic systems can indeed lead to unexpected generalization \citep{shokar2025deeplearningevolutionoperator}, highlighting the need to pre-train foundational models on datasets with chaotic dynamics \citep{herde2024poseidon}. While our theory proves that a finite number of trajectories is insufficient for complete discovery, this need not contradict empirical successes: performance of these models is often quantified distributionally. That is, they perform well on test data similar to their training data, without the expectation of true, out-of-distribution generalization.

\paragraph{Broader Implications:}
The findings of this article provide a rigorous mathematical statement: the inherent nature of a physical system, specifically, whether it is chaotic or not, is a critical factor in determining the potential success of purely data-driven modeling. This may help explain one of the most notable successes of large-scale machine learning in dynamics prediction, in \emph{climate and weather modeling} and \emph{computational fluid dynamics} \citep{vinuesa2022enhancing, price2025gencast, kochkov2021machine,Kutz_2017}. These domains are governed by inherently chaotic systems, ensuring their dynamics are, in principle, \emph{discoverable from observational data}, which in turn promotes generalization.

It would, of course, be unscientific to claim that chaos is the sole reason \textemdash{} success is shaped by many factors, including data availability, access to classical modeling tools, and domain knowledge \citep{raghu2020survey, uren2020critical, roscher2020explainable}. We hypothesize, however, that the high informational content embedded within trajectories of chaotic systems is one of the primary drivers. The importance of this factor is starkly illustrated by the challenges these methods face in engineering domains, where chaos is deliberately avoided.

For instance, in the development of digital twins, systems are designed for controllability, stability and predictability \textemdash{} the very opposite of chaos \citep{wagg2025philosophical,eftimie2023digital}. According to our findings, the data from such non-chaotic systems is insufficient to uniquely identify the governing laws. A machine learning or symbolic model may perfectly fit the available data, yet fail to generalize, as it represents only one of many possible models consistent with the limited observations. This reveals a critical conundrum: the systems most amenable to discovery are often the least desirable for control and engineering applications. 

The path forward, however, is not to abandon data-driven methods. Instead, we must recognize their limits and move towards hybrid, knowledge-informed data-driven approaches. Integrating known physical priors, such as conservation laws, is not merely a way to improve data efficiency \citep{brehmer2025does}, but is a mathematical necessity for recovering well-posedness in discovery of non-chaotic systems.

\section*{Acknowledgments}
Authors would like to thank Olga Obolenets, Lucy Richman, Sanjeeb Dash and Karan Srivastava for useful discussions and feedback.
\paragraph*{Funding:}
ZS acknowledges support from the Cantab Capital Institute for the Mathematics of Information, Christs College and the Trinity Henry Barlow Scholarship scheme. 
CBS acknowledges support from the Royal Society Wolfson Fellowship, the EPSRC advanced career fellowship EP/V029428/1, the EPSRC programme grant EP/V026259/1, the Wellcome Innovator Awards 215733/Z/19/Z and 221633/Z/20/Z, the EPSRC funded ProbAI hub EP/Y028783/1, and the European Union Horizon 2020 research and innovation programme under the Marie Skodowska-Curie grant agreement REMODEL.
This work of G. K. was supported in part by the Konrad Zuse School of Excellence in Reliable AI (DAAD), the Munich Center for Machine Learning (BMBF) as well as the German Research Foundation under Grants DFG-SPP-2298, KU 1446/31-1 and KU 1446/32-1. Furthermore, P.S. and G. K. acknowledge support from LMUexcellent, funded by the Federal Ministry of Education and Research (BMBF) and the Free State of Bavaria under the Excellence Strategy of the Federal Government and the L\"ander as well as by the Hightech Agenda Bavaria and additional support by the project ”Genius Robot” (01IS24083), funded by the Federal Ministry of Education and Research (BMBF).
\paragraph*{Author contributions:}
Z.S. proposed the idea and led the research;  Z.S. and P.Z. developed and executed the idea. \textit{(Discoverability Section)}: Z.S., P.Z., and P.S. collaboratively wrote and developed the tools for this section. \textit{(Continuous Discoverability)}: Z.S., P.Z., and P.S. collaboratively wrote and developed the tools for this section. \textit{(Analytical Discoverability)}: P.Z. and Z.S. developed the analytical discoverability results. \textit{(Stability \& Experimental Results)}: Z.S. and L.H. developed the stability and experimental results. \textit{(Conservation Laws Results)}: Z.S. and P.S. developed the theoretical results for this section. \textit{Manuscript Preparation}: All authors contributed to the writing of the manuscript. \textit{Supervision}: CB.S., L.H., and G.K. supervised and guided the project.

\bibliographystyle{plainnat}
\bibliography{references}  

\newpage
\appendix
\appendixpage

\startcontents[sections]
\printcontents[sections]{l}{1}{\setcounter{tocdepth}{2}}

\vspace{2em}
\noindent This supplementary material provides a broader, more detailed exposition of the mathematical setup and theoretical results behind material presented in the main article. 
\section{Related Works}

The primary motivation for this work is the recent surge of interest in methods for approximate discovery of dynamical systems directly from trajectory data. This fundamental problem, however, is not novel. Scientists have manually derived models from empirical observations for centuries \citep{1609kepler,newton1822philosophiae,fourier1822theorie,ohm1827galvanische,fick1855ueber}, while from a computational perspective, it represents a long-standing area of research with foundational contributions dating back several decades \citep{ho1966effective,bellman1970structural,langley_bacon_mind1981,langley1987scientific,80202,bongard2007automated,schmidt2009distilling,6213241}.

This research area focuses on learning either the underlying vector field, $F$, or its corresponding flow map, $\phi_F$ directly from data. Methodologically, these approaches can be broadly classified into two main paradigms: surrogate learning, which approximates the dynamics using black-box models, and symbolic discovery, which aims to find explicit, interpretable mathematical expressions \citep{yu2024learning,north2023review}.
\paragraph{Learning differential equations.}
To address systems with unknown or computationally expensive physical models, a variety of surrogate methods have been developed. While foundational approaches include classical techniques like reduced-order modeling \citep{kerschen2005method,1102568,schilders2008model,benner2015} and Dynamic Mode Decomposition (DMD) \citep{mezic2005spectral, schmid2010dynamic,brunton2017chaos,kutz2016dynamic,williams2015data},
the field has increasingly shifted towards deep learning. Initially, these contemporary methods focused on parameterizing as neural networks either the governing vector field \citep{raissi2018deep, chen2018neural, greydanus2019hamiltonian, cranmer2020lagrangian} or the flow map \citep{bilovs2021neural, canizares2024symplecticneuralflowsmodeling}. Within the last five years, however, the predominant focus in scientific machine learning has shifted towards Neural Operators, which aim to learn the evolution operator of a physical system directly in function space \citep{luDeepONetLearningNonlinear2021,JMLR:v24:21-1524}. Among these, the Fourier Neural Operator has become a foundational model, leveraging Fourier transforms for efficient global convolutions \citep{li_fourier_2021}. While theoretically framed in infinite-dimensional function spaces, practical inference of these models is performed in finite dimensions, meaning they do not formally constitute operators \citep{bartolucci2023representation}. Neural Operators have become widely adopted, powering applications in weather forecasting \citep{pathak_fourcastnet_2022}, fluid dynamics \citep{renn_forecasting_2023}, seismic imaging \citep{li_solving_2023}, carbon capture \citep{wen_real-time_2023}, and plasma modeling for nuclear fusion \citep{gopakumar_plasma_2024}.

Despite their powerful approximation capabilities, the black-box nature of models like neural operators presents a fundamental challenge to their reliability and trustworthiness \citep{vafa2025foundationmodelfoundusing}. This lack of interpretability and generalizability is one of the main critical barriers to their adoption in scientific and engineering domains. Motivated by this, one line of research extends beyond neural models by discovering physically interpretable symbolic expressions.

\paragraph{Symbolic recovery of differential equations.}

Symbolic regression \citep{la2021contemporary} seeks to uncover analytical relationships between variables \(x_1, \ldots, x_n\) and \(y\) without assuming a fixed model structure. Early methods relied on genetic programming \citep{augusto2000symbolic, schmidt2009distilling, Schmidt2010AgefitnessPO, cranmer2023interpretable}, later extended with reinforcement learning \citep{petersen2019deep, sun2022symbolic} and hybrid approaches \citep{mundhenk2021symbolic}. Physical priors have also been introduced, such as exploiting symmetries \citep{udrescu2020tegmark} or embedding scientific laws \citep{cornelio2023combining,cory2024evolving,srivastava2025ai}.  
To overcome discrete search limitations, several methods reformulate the problem, enabling continuous optimization using linear \citep{McConaghy2011, brunton2016discovering} or compositional basis \citep{martius2016extrapolation, sahoo2018learning, scholl2025parfam} functions. Recent transformer-based approaches infer symbolic expressions directly from data, either through pre-training on synthetic corpora \citep{biggio2021neural, kamienny2022end, holt2022deep} or leveraging large language models \citep{grayeli2024symbolic, shojaee2024llm}.  
These techniques can be directly applied to discover differential equations by treating derivatives as variables \citep{la2021contemporary}. The Sparse Identification of Nonlinear Dynamics (SINDy) framework \citep{brunton2016discovering} uses sparse regression to infer closed-form ODEs and has been extended to PDEs \citep{Rudy2017DatadrivenDO}, implicit systems \citep{kaheman2020sindy}, learned feature spaces \citep{champion2019}, and non-smooth dynamics \citep{quade2018sparse}. Beyond SINDy, neural surrogates \citep{Hasan2020, both2021choudhury, stephany2022pde, chen2021physics} and variational formulations \citep{qian2022dcode, kacprzyk2023dcipher} address instability in numerical differentiation.

\paragraph{Identifiability.} Traditionally, data-driven discovery of dynamical systems assumes a known structural form of the underlying equations, with parameters estimated from data \citep{alessandrini1986identification, acar1993identification, knowles2001parameter}. The well-posedness of such inverse problems has been widely studied, particularly with respect to parameter identifiability \citep{bellman1970structural, cobelli1980parameter, distefano1980parameter, miao2011identifiability, nguyen1982review}. In contrast to these classical approaches, which focus on determining a finite set of parameters, our work seeks to identify an entire function within a broad, non-parametric class of dynamical systems, following the advances in symbolic recovery of differential equations outlined above.
Recent progress toward this direction includes the identifiability results for linear dynamical systems by \citet{stanhope2014identifiability} and \citep{qiu2022identifiability}, as well as the exploration by \citet{casolo2025identifiabilitychallengessparselinear}, who relate identifiability to the sparsity of the system matrix and show that many such systems are unidentifiable in practice. Related recovery guarantees exploiting mixing in Lorenz-type dynamics have previously been obtained for finite-dimensional polynomial models \citep{tran2017exact}. As these approaches are limited to linear systems, the most relevant prior work to ours is due to \citet{scholl2022symbolicrecoverydifferentialequations} and \citep{scholl2023uniqueness}, who established necessary and sufficient conditions for identifiability in more general, non-linear classes of ODEs and PDEs, later extended to noisy systems by \citet{hauger2024robustidentifiabilitysymbolicrecovery}. Building on these foundations, we extend the framework of \citet{scholl2022symbolicrecoverydifferentialequations} to generic dynamical systems and demonstrate that identifiability in this setting is fundamentally determined by the presence of chaos\textemdash{}an aspect that, to our knowledge, has not been previously examined.

\section{Mathematical Background}\label{sec:mathback}
In what follows we concern ourselves with dynamical systems of the form 
\begin{equation}\label{eq:basic}
\dot u = F(u) \quad\text{ for } u:\mathbb{R}\rightarrow\mathbb{R}^d    
\end{equation}
The aim of this study is to consider when symbolic or parametric discovery of $F$ based on observation data can be considered well-posed from the point of view of uniqueness. In order to define chaos, we will need to consider the flows generated by such a dynamical system. For this, assume that $F$ is locally Lipschitz. By the Picard-Lindel\"{o}f theorem there is locally a unique solution to the dynamical system. This may not extend to infinite time, so we make the following definition. A flow domain is an open subset $D \subseteq \mathbb{R}^d \times \mathbb{R}$ such that each time slice $D^x = \{t \in \mathbb{R} | (x, t) \in D\}$ is an open interval containing zero. Then it is a classical result, see \citep[Chapter 9]{Lee2012} for the smooth case, that we may find a maximal (with respect to inclusion) flow domain $D \subseteq \mathbb{R}^d \times \mathbb{R}$ and a flow-map $\phi_F : D \rightarrow \mathbb{R}^d$, which solves the initial value problem 
\begin{equation}
    \frac{d}{dt} \phi_F(x_0, t) = F(\phi_F(x_0, t)), \quad \phi_F(x_0, 0) = x_0
\end{equation}
We will denote the time $t$ flow map as $\phi^t(\cdot)=\phi(\cdot,t)$, this is only defined on the open set $D^t = \{x \in \mathbb{R}^d : (x, t) \in D\}$.

In general this flow may not be extendable to infinite time, either in the forward or backward direction. If one has a compact subset $V \subseteq \R^d$ which is forward invariant in the sense that for all $t > 0$, $\phi^t(V) \subseteq V$, then any flow starting from $x \in V$ can be extended to $+\infty$. This is a variant of the argument of \citet[Theorem 9.16]{Lee2012}. Every $p \in V$ has an open neighborhood $U_p$ on which we may extend the flow by $\varepsilon_p > 0$, and by the forward invariance assumption each of these flows remains in $V$. By compactness of $V$, choose a finite subcover $U_{p_1}, \dots, U_{p_r}$, and note that then every flow at a point in $V$ can be extended by $\varepsilon = \min_i \varepsilon_{p_i}$ and remains contained in $K$. Repeated extension by $\varepsilon$ then gives the extension to $+\infty$. 

Similarly, we see that if $V$ is compact and backward invariant in the sense that for $t < 0$ we have $\phi^t(V) \subseteq V$ then every flow beginning in $V$ will extend to $-\infty$. Combining the two, if $V$ is compact and both forward and backward invariant then all flows beginning in $V$ will extend to $(-\infty, +\infty)$.
\subsection{Identifiability and Discoverability}
\begin{definition} [Uniqueness/Identifiability] \label{def:unique-dynamical-system}
    Let $u:U\rightarrow\mathbb{R}^d$ be a differentiable function on the open set $U\subset\mathbb{R}$. Further, let $V$ be a set of functions which map from $\mathbb{R}^d$ to $\mathbb{R}^d$ and $F\in V$ such that
    \begin{equation} \label{eq:def-unique-dynamical-system}
       \dot u = F(u).
    \end{equation}
    We say that the function $u$ \emph{solves a unique dynamical system} if $F$ is the unique function in $V$ such that \Cref{eq:def-unique-dynamical-system} holds. In such cases $F$ is said to be \emph{identifiable} from $u$.
\end{definition}
As one should expect, identifiability inherently depends on the trajectory $u$, and not all trajectories are sufficient to result in uniqueness. Because of this, we instead consider uniqueness under \emph{some} set of trajectories, which we term \emph{discoverability}.
\begin{definition} [$n$ Discoverability] \label{def:ndiscoverable-dynamical-systems} 
    We say that $F$ is discoverable from $n$ trajectories in $V$, if there exist $u_1,\dots, u_n:U\rightarrow\mathbb{R}^d$ differentiable trajectories on the open set $U\subset\mathbb{R}$, such that $F$ is the unique function in $V$ such that \Cref{eq:ndiscoverable-dynamical-systems} holds for all $n$ trajectories 
    \begin{equation} \label{eq:ndiscoverable-dynamical-systems}
       \dot u_i = F(u_i).
    \end{equation}
    We say that $F$ is finitely discoverable in $V$, if there is $M\in\mathbb{N}$ s.t. $F$ is $M$ discoverable.
\end{definition}
\begin{nb}
    The definition of discoverability for us is practically equivalent to that of uniqueness\textemdash{}primarily as a hallmark of reliability and well-posedness. If one were to try to learn a surrogate or derive a symbolic representation for $F$, any method that is able to fit the data would necessarily discover the right system by uniqueness. 
\end{nb}
In this paper, we primarily focus on chaotic systems due to their discoverability properties. The study of chaotic systems spans many different fields, including physics, computer science, and mathematics, with various definitions capable of capturing different aspects of what we intuitively regard as chaotic behavior, without an all-encompassing definition. The main component, without which no system would ever be considered chaotic, is that of \emph{topological transitivity} \citep{blanchard2008topologicalchaosmean}.
\begin{definition}[Topological Transitivity]
The dynamical system $F$ with corresponding flow $\phi^t$ is \emph{topologically transitive} if for any non-empty, open sets $U, V \subset \mathbb{R}^d$, $\exists\, t > 0$ such that:
    \begin{equation}
        \phi^t(U) \cap V \neq \emptyset
    \end{equation}
\end{definition}
In the following sections, the main object of study will be the trajectories generated by flowing the dynamical system. 
\begin{definition}[Trajectory] We define the \emph{trajectory} or \emph{orbit} as:
\[
\operatorname{Traj}(x, F):=\left\{\phi^t(x): t \in [0,+\infty)\right\} .
\]
We will say that $\phi$ is \emph{hypercyclic} if it has a dense trajectory, that is, there exists $x \in M$ such that $\operatorname{Traj}(x, F)$ is dense in $M$; such an $x$ will be called a \emph{hypercyclic point} for $\phi$.
\end{definition}

The main result we shall rely on is Birkhoff's theorem on topological transitivity \citep{birkhoff_surface_1922}, which we state here for the setting of flows: 

\begin{lemma}[Continuous Birkhoff's Lemma]\label{lem:birk}
     Let $X$ be a complete, separable, and connected metric space with a continuous dynamical system $\phi$. Further assume that the flow $\phi$ is defined for all time. Then $\phi$ is topologically transitive if and only if it has a hypercyclic point. In this case, the set of hypercyclic points is dense. 
\end{lemma}
\begin{proof}
    First suppose that $\phi$ has a hypercyclic point, call it $x$. We claim that not only is the trajectory of $x$ dense, but each tail of the trajectory is dense, where the tail after time $T > 0$ is defined by
    \begin{equation*}
        \operatorname{Tail}(x,T) = \{\phi^t(x) : t \in [T, +\infty)\}
    \end{equation*}
    Note that by the assumption that the flow is defined for all time, each map $\phi^t$ is a homeomorphism of $X$ with inverse $\phi^{-t}$. Thus it commutes with taking closures and we obtain,
    \begin{equation*}
        \overline{\operatorname{Tail}(x,T)} = \phi^T(\overline{\operatorname{Traj}(x,F)}) = X
    \end{equation*}
    
    Fix $U, V$ open sets. By density of $\operatorname{Traj}(x,F)$ we may find $t > 0$ such that $\phi^t(x) \in U$. Since the tail after $t$ is dense, we may find $t' \geq t$ such that $\phi^{t'}(x) \in V$. Then
    \begin{equation*}
        \phi^{t' - t}(\phi^{t}(x)) = \phi^{t'}(x) \in V
    \end{equation*}
    Therefore $\phi^{t' - t}(U) \cap V \neq \varnothing$.
    
    Suppose then that $\phi$ is topologically transitive. Let $(V_j)_{j \in \mathbb{N}}$ be a countable basis of open sets for $X$. Let $H$ be the set of hypercyclic points. Then we may write
    \begin{equation}
        H = \bigcap_j \bigcup_{t > 0} \phi_t^{-1}(V_j)
    \end{equation}
    Indeed, suppose that $x$ is in the set on the right. Let $y \in X$ and let $U$ be an open neighborhood of $y$. As $(V_j)$ is a basis for the topology on $X$ we may choose $j_0$ such that $y \in V_{j_0} \subseteq U$. By definition, $x \in \bigcup_{t > 0} \phi_t^{-1}(V_{j_0})$, hence there is some $t_0$ such that $x \in \phi_t^{-1}(V_{j_0})$. Therefore $\phi_t(x) \in V_{j_0}$. Thus we see that the orbit of $x$ is dense and hence $x \in H$. Conversely if $x \in H$, then for any $j$, we may find $t$ such that $\phi_t(x) \in V_{j}$, implying that $x \in \bigcup_{t > 0} \phi_t^{-1}(V_j)$. Therefore $x$ is contained in the set on the right and we have the desired equality.
    
    Now define $H_j = \bigcup_{t > 0} \phi_t^{-1}(V_j)$. We claim that this set is dense. Let $y \in X$ and let $U$ be an open neighborhood of $y$. As above, we may shrink $U$ and assume that $U = V_{j_0}$ for some $j_0$. By topological transitivity, there is some $t_0 > 0$ such that $\phi_{t_0}(V_{j_0}) \cap V_j$ is non-empty. Hence $V_{j_0} \cap \phi_{t_0}^{-1}(V_j)$ is non-empty, therefore $V_{j_0} \cap H_j$ is non-empty. Therefore $H_j$ is dense. Now by the Baire category theorem, as $H$ is the countable intersection of dense sets, we get that $H$ is dense as well. 
\end{proof}

\begin{remark}
    In the above result the direction that $\phi$ is topologically transitive implies it has a hypercyclic point does not require the assumption that the flow $\phi$ is defined for all time. In general one can weaken this assumption to only that $\phi^t$ being surjective for all $t > 0$ but this is a much less natural condition from the perspective of dynamical systems.
\end{remark}
\subsection{Chaos}
Instead of the usual definition of chaos, we will consider a general definition for systems, which may not be chaotic everywhere, but instead only on an attractor (\Cref{def:attractor}). For systems, that are chaotic on the whole domain, the two definitions agree, as in such settings the whole domain is an attracting set. Oftentimes, these are also termed \textit{strange} based on the Hausdorff dimension of the attracting set. We follow \citet{crob} in defining these. 
\begin{definition}[Invariant set]
An \emph{invariant set} for a map ${\phi_{F}}$ is a set $\mc{A}$ in the domain such that ${\phi_{F}}(\mc{A})\subseteq\mc{A}$. Therefore, for every ${x}$ in $\mc{A}$, the trajectory $\Orb{(x, F)}$ is entirely contained in $\mc{A}$.   
\end{definition}

An invariant set $\mc{A}$ is \emph{topologically transitive} provided that there is a point $x^*$ in $\mc{A}$ such that the trajectory of ${x}^*$ is dense in $\mc{A}$, i.e., $\operatorname{cl}\left(\Orb\left({x}^*, {F}\right)\right)=\mc{A}$. Another important feature of chaotic systems is sensitivity, which has little importance for the upcoming results of ours. 
\begin{definition}[Sensitivity]
 A map ${\phi_{F}}$ has \emph{sensitive dependence on initial conditions (SDIC)} at all points of $\mc{A}$ provided that for each point ${x} \in \mc{A}$ there is an $r>0$ such that for all $\delta>0$, there are ${y}$ and $t \geq 0$ with $\|{y}-{x}\| \leq \delta$ and $\left\|{\phi}_F^t({y})-{\phi}_F^t({x})\right\|>r$.   

Compared to general sensitivity, this effectively means that the nearby point ${y}$ whose trajectory moves away from the trajectory of ${x}$ can be chosen in the set $\mc{A}$ and not just in the ambient space.
\end{definition}

\subsubsection{Attractors}\label{sec:attractors}
The following definition of an attractor, as in \citet{robinson1998dynamical, robinson2012introduction}, is similar to asymptotic stability of a fixed point, and is defined in terms of a trapping region.
\begin{definition}[Trapping region]
A set ${U}$ is called a \emph{trapping region} ${U}$ for map ${f}$ provided its closure is compact and
$$
\operatorname{cl}({\Orb (U,F)}) \subset \operatorname{int}({U}).
$$
Since the closure is mapped into the interior, the set is mapped well inside itself. A set $\mc{A}$ is called an \emph{attracting set} if there exists trapping region ${U}$ such that
$$
\mc{A}=\bigcap_{t\in[0,+\infty)} {\phi}^t({U}).
$$
The largest such $U$ is called the \emph{basin of attraction} for $\mc{A}$.
Because a trapping region has compact closure, an attracting set is necessarily compact.
\end{definition}

\begin{definition}[Attractor]\label{def:attractor}
A set $\mc{A}$ is called an \emph{attractor} for ${\phi_{F}}$ if it is an attracting set such that there are no nontrivial sub-attracting sets, i.e., no attracting set $\mc{A}^{\prime} \subset \mc{A}$ such that $\emptyset \neq \mc{A}^{\prime} \neq \mc{A}$.   
\end{definition}

\begin{definition} [Chaos]
A map ${f}$ is said to be \emph{chaotic} on an invariant set $\mc{A}$ provided that the following conditions are satisfied:
\begin{itemize}
    \item[(a)] The map $f$ has sensitive dependence on initial conditions when restricted to the attractor $\mc{A}$.
    \item[(b)] The map ${f}$ is topological transitive on $\mc{A}$, i.e., there exists an ${x}^*$ such that $\operatorname{cl}\left(\Orb{\left({x}^*, {f}\right)}\right)=\mc{A}$.
\end{itemize}
A set $\mc{A}$ is called a \emph{chaotic attractor} for a map ${f}$ provided that the set $\mc{A}$ is an attractor for ${f}$ and ${f}$ is chaotic on $\mc{A}$.
\end{definition}
An alternative way to present ideas above is through $\omega$-limit sets of trajectories.
\begin{definition}
    The $\omega$-limit set of the trajectory $\gamma$ is defined as
    \begin{equation*}
        \lim_{\omega} \gamma = \bigcap_{s \in [0,\infty)} \overline{\{\gamma(t) \mid t > s\}}
    \end{equation*}
    A generalized limit cycle is a trajectory $\Omega$ such that $\lim_{\omega} \Omega = \overline{\Omega}$. We shall call both the trajectory $\Omega$ and the closure $\overline{\Omega}$ a generalized limit cycle, as from a discoverability perspective these are the same.
\end{definition}

We now give certain standard properties of these $\omega$-limit sets. It is possible for the $\omega$ limit to be empty, but if the trajectory $\gamma$ is bounded then it is necessarily non-empty. 

\begin{proposition}\label{prop:PropertiesOfOmegaLimits}
    \begin{enumerate}
        \item[(i)] If $\gamma$ is periodic then $\lim_\omega \gamma = \gamma$.
        \item[(ii)] $\lim_{\omega} \gamma$ is invariant under the flow.
        \item[(iii)] Either $\lim_{\omega} \gamma \cap \gamma = \varnothing$ or $\overline{\gamma} = \lim_{\omega} \gamma$.
    \end{enumerate}
\end{proposition}
\begin{proof}
    For $(i)$ it is easy to see that by periodicity each set $\{\gamma(t) \mid t > s \}$ is actually equal to $\gamma$. Now we use that because $\gamma$ is periodic, it is the image of $[0, 1]$ (potentially after rescaling), hence is closed. Thus $\lim_\omega \gamma = \gamma$.

    For $(ii)$ fix $x \in \lim_\omega \gamma$. Then there is a sequence $t_n \rightarrow T$ such that $\gamma(t_n) \rightarrow x$. But then we have, for any $t \in \mathbb{R}$, $\phi^t(\gamma(t_n)) = \gamma(t_n + t) \rightarrow \phi^t(x)$. Here if $t < 0$ we take $n$ large enough such that $t_n + t > 0$. This shows that $\lim_{\omega} \gamma$ is both forward and backward invariant under the flow $\phi$.

    For $(iii)$ assume that $\lim_{\omega} \gamma \cap \gamma$ is non-empty. Let $x = \gamma(t)$ be in this intersection. Then by part $(ii)$ as $\lim_{\omega} \gamma$ is invariant under the flow we see that the forward and backward orbit of $x$ is contained in $\lim_{\omega} \gamma$. But this orbit is just $\gamma$. Finally, as $\lim_{\omega} \gamma$ is closed, we have that $\overline{\gamma}$ is contained inside of it. But it is clear from the definition of $\lim_{\omega} \gamma$ we have $\lim_{\omega} \gamma \subseteq \overline{\gamma}$.
\end{proof}

There are three qualitatively different behaviors for $\lim_{\omega} \gamma$. First, we have the case that $\lim_{\omega} \gamma$ is not a limit cycle. Then the two cases are determined by whether the trajectory $\gamma$ is closed.

\begin{proposition}\label{prop:ClosedTrajectoryWithLimitsIsPeriodic}
    Assume that $\lim_{\omega} \gamma = \overline{\gamma}$. Then if $\gamma$ is closed, $\gamma$ is periodic.
\end{proposition}
\begin{proof}
    As $\gamma$ satisfies a differential equation, to show that $\gamma$ is periodic we simply need to a find two times $t_0, t_1$ with $\gamma(t_0) = \gamma(t_1)$. We will show that there must be a time $t_0 > 0$ such that $\gamma(t_0) = x_0$. By assumption $\lim_{\omega} \gamma = \gamma$, hence $x_0 \in \lim_{\omega} \gamma$. This implies that for any $s \in \mathbb{R}_{\geq 0}$, $x_0 \in \overline{\{\gamma(t) \mid t > s\}}$. Hence we may find a strictly increasing sequence $t_n$ such that $\gamma(t_n) \rightarrow x_0$. If the sequence $t_n$ diverges to infinity then we must have $F(\gamma(t_n)) \rightarrow 0$. Hence $F(x_0) = 0$ so $x_0$ is a fixed point, and hence $\gamma$ is trivially periodic. Otherwise $t_n$ converges to some $T$. Then $\gamma(T) = x_0$ as required so $\gamma$ is again periodic.
\end{proof}

Hence we obtain that if $\gamma$ is a closed trajectory then it is either periodic or it diverges to infinity. 
For our main purposes we only need that $\lim_{\omega} \gamma \subseteq \overline{\gamma}$, which follows from the definition. If $\gamma$ is contained in the basin of attraction of some attractor then $\lim_\omega \gamma$ will be this attractor. In general though, the $\omega$-limit set behaves like the attractor of the specific trajectory $\gamma$. 

\subsection{Analytic Geometry}
The central challenge of discoverability within the space of analytic functions, which forms the core of this work, has to be approached through the lens of analytic geometry. We begin with the basic object of real analytic geometry. 
\begin{definition}
    A closed subset $X \subseteq \mathbb{R}^d$ is called an analytic subset if for every point $x \in X$ there is an open neighborhood $U$ of $x$ such that $X \cap U$ is given by the common zero set of finitely many analytic functions $f_1, \dots, f_n$ on $U$. 
\end{definition}

For our purposes, we would like the sets we consider to be cut out not just locally by analytic functions but actually cut out by global analytic functions. This brings us to the following notion:

\begin{definition}
    A closed subset $X \subseteq \mathbb{R}^d$ is C-analytic if there are finitely many real-analytic function $f_1, \dots, f_n$ such that $X$ is the common zero-set of these functions. 
\end{definition}

These two definitions are related, but not equivalent, which is the one important difference between complex analytic geometry and real analytic geometry. Indeed, there exist analytic sets in $\mathbb{R}^d$ that are not C-analytic, for example \citep[Examples 1.2.26]{Acquistapace2022}. 

There are two related weaker notions of analyticity that we shall need.

\begin{definition}
    A set $X \subseteq \mathbb{R}^d$ is semi-analytic if for each point $x \in \mathbb{R}^d$ there is an open neighborhood $U$ of $x$ and analytic functions $f_1, \dots, f_n$ and $g_1, \dots, g_m$ on $U$ such that
    \begin{equation*}
        X \cap U = \{y \in U \mid f_i(y) = 0, g_j(y) > 0 \quad \forall i, j\}
    \end{equation*}

    The set $X$ is C-semianalytic if the functions $f_1, \dots, f_n$ and $g_1, \dots, g_n$ can be chosen to be analytic on $\mathbb{R}^d$. The set $X$ is a tame C-semianalytic set if it is a finite union of C-semianalytic sets of the form $Z \cap U$ where $Z$ is a C-analytic set and $U$ is an open C-semianalytic set.
\end{definition}

\begin{definition}
    A set $X \subseteq \mathbb{R}^d$ is subanalytic if it is locally the projection of a semi-analytic set. Namely, for each $x \in \mathbb{R}^d$ there exists an open neighborhood $U$ of $x$, a natural number $n$, and a relatively compact semi-analytic set $A \subseteq U \times \mathbb{R}^n$ such that $\pi(A) = X \cap U$ where $\pi : \mathbb{R}^d \times \mathbb{R}^n \rightarrow \mathbb{R}^d$ is the projection.
\end{definition}

\subsection{Germs of Analytic Functions}
We introduce briefly the notion of germs of analytic function. These should be thought of as local analytic functions at a point where we forget about exactly what neighborhood our function lives on. 

\begin{definition}
    Let $x \in \mathbb{R}^d$ be a point. Then a germ of an analytic function at $x$ is a pair $(U,f)$ where $U$ is an open set containing $x$ and $f \in C^{\omega}(U)$. We say that two such pairs $(U,f)$ and $(V,g)$ are equivalent if
    \begin{equation*}
        f|_{U \cap V} = g|_{U \cap V}
    \end{equation*}
    Write $\mathcal{O}_{\mathbb{R}^d, x}$ for the set of equivalence classes of germs of analytic functions at $x$. This naturally forms a ring by adding and multiplying functions.
\end{definition}

Then we have the Weierstrass Preparation Theorem, for more of an introduction see \citep{Acquistapace2022}.

\begin{theorem}\label{thm:WeierstrassPreperation}
    The ring $\mathcal{O}_{\mathbb{R}^d, 0}$ is Noetherian.
\end{theorem}

\section{Discoverability}\label{sec:sets_of_uniqueness}
With all the definitions now introduced, we can move onto the main topic of this study: discoverability. The main result, turning the problem of discoverability into a much simpler one is based on \citet[Proposition 1]{scholl2022symbolicrecoverydifferentialequations}, extended here to multi-dimensional systems.
\begin{proposition} \label{prop:general-uniqueness}
    Let $V$ be any class of functions mapping from $\mathbb{R}^d$ to $\mathbb{R}^d$ which is closed under addition and subtraction. Assume there exists $F\in V$ such that $\dot u = F(u)$. Denote with $\mathcal{O}=u(U)$ the image of $u$. Then, $F$ is the unique function in $V$ such that $\dot u = F(u)$ if and only if there is no $G\in V\setminus\{0\}$ which is $0$ on whole $\mathcal{O}$, i.e., $G|_\mathcal{O}\equiv0$.    
\end{proposition}
\Cref{prop:general-uniqueness} allows us to turn the hard problem of uniqueness into a simpler problem, of establishing when a given trajectory covers a \emph{set of uniqueness} for a given function class, see examples in \Cref{fig:sofu}:
\begin{definition}[Set of Uniqueness]\label{def:Concrete-Def-Uniqueness}
    Let $V$ be any class of functions mapping from $\mathbb{R}^d$ to $\mathbb{R}^d$. A set $A$ is called a \emph{set of uniqueness} for $V$, if for all $ G\in V$, $G|_A=0$ implies $G|_{\mathbb{R}^d}=0$.
\end{definition}

One can also rephrase the set of uniqueness question using the sheaf theory, which allows it to generalize to other spaces and other classes of functions. We leave this to future work.

\subsection{The Continuous Case}
In the continuous case, uniqueness becomes almost trivial, as sets of uniqueness are those dense in the whole space. See e.g. \citep[Theorem 4]{scholl2022symbolicrecoverydifferentialequations}. We provide a proof here for conciseness.
\begin{lemma}\label{prop:cont_uniq}
Let $F:\mathbb{R}^d\rightarrow\mathbb{R}^d$ be a continuous function and $\mathcal{D}\subset\mathbb{R}^d$. Then $F|_\mathcal{D}=0$ implies $F=0$ if and only if $\mathcal{D}$ is dense. 
\end{lemma}
\begin{proof}
     If $\mathcal{D}$ is dense, then using that continuous functions which agree on a dense set agree on the whole space, we get that $F = 0$. 

     Conversely, if $\mathcal{D}$ is not dense, then there must exist an open set $U \subseteq \R^d$ such that $\mathcal{D} \cap U = \varnothing$. Take some open ball $B(x, r) \subseteq U$. Then we can construct a continuous function $F$ with $F|_{\R^d \backslash U} = 0$ but $F|_{B(x,r)} \neq 0$. As an explicit example one can take
     \begin{equation}\label{eq:ContinuousBumpFunction}
         F(y) = \begin{cases}
             1 - \frac{|y - x|}{r} & y \in B(x, r) \\
             0 & y \notin B(x,r)
         \end{cases}
     \end{equation}
\end{proof}
\begin{nb}
    We note that this result can be extended in a similar manner to any class of functions which has bump functions or equivalently partitions of unity. More concretely, by this we mean that for any open set $U$ we may find a non-zero function $F : \mathbb{R}^d \rightarrow \mathbb{R}$ in the given class, such that $F|_{\mathbb{R}^d \backslash U} = 0$. This includes most notably any class of functions containing compactly supported smooth functions, though there are many more examples. Then the same exact argument as in the continuous case gives the desired statement for a general class of functions which has bump functions. \Cref{eq:ContinuousBumpFunction} is a construction of such a bump function for the class $C^0$. For an example of a smooth compactly supported such bump function, see \citep[Lemma 2.22]{Lee2012}.
\end{nb}
\subsection{The Analytic Case}

In the case of real analytic functions, the problem becomes increasingly more complicated. As we discuss below, there exist many sufficient conditions for a set to be a set of uniqueness, of which one will be useful for establishing uniqueness for dynamical systems, like the Lorenz system. To the best of our knowledge there exists no geometric equivalence statement to tell us when a given set is a set of uniqueness, as pathological examples exist, see e.g. \citep{lebl2021examplecompactnoncanalyticreal}. Choosing $V$ as the set of analytic functions, the following result gives us, similarly to \citet{scholl2022symbolicrecoverydifferentialequations}, a sufficient criterion for uniqueness:
\begin{proposition} \label{prop:measure-condition-for-uniqueness-of-analytic-functions}
    Let $F:\mathbb{R}^d\rightarrow\mathbb{R}^d$ be an analytic function and $\mathcal{D}\subset\mathbb{R}^d$ a set with $\lambda^d(\mathcal{D})>0$, where $\lambda^d$ is the $d$-dimensional Lebesgue-measure. Then $F|_\mathcal{D}=0$ implies $F=0$. I.e., sets with positive Lebesgue measure are sets of analytic uniqueness.
\end{proposition}
A simpler version directly following from \Cref{prop:measure-condition-for-uniqueness-of-analytic-functions}, is the well-known Identity Theorem.
\begin{theorem} [Identity Theorem]\label{thm:open-condition-for-uniqueness-of-analytic-functions}
    Let $F:\mathbb{R}^d\rightarrow\mathbb{R}^d$ be an analytic function and $\mathcal{D}\subset\mathbb{R}^d$ is open. Then $F|_\mathcal{D}=0$ implies $F=0$. I.e., sets that contain an open set are sets of analytic uniqueness.
\end{theorem}
\citep{scholl2022symbolicrecoverydifferentialequations} use this statement to show that uniqueness follows from the image $\mathcal{D}$ containing an open set or, as a consequence, the Jacobian of $g$ (in our case simply $u$) having rank $d$. These statements, however, are not useful in our settings, apart from the case $d=1$. If $d=1$, then $\mathcal{D}\subset\mathbb{R}$ contains an open set if and only if $u$ is not constant, making uniqueness in this case trivial. From here onward we will now assume that $d > 1$.

Let us consider as an example the classical Lorenz system \citep{tucker1999lorenz}. It is a 3-dimensional dynamical system, and both the measure criterion \citep[Corollary 2]{scholl2022symbolicrecoverydifferentialequations} and the Jacobian criterion \citep[Theorem 3]{scholl2022symbolicrecoverydifferentialequations} (since the Jacobian of $u$ needs to have rank 3, but $u$ maps from $\mathbb{R}$ to $\mathbb{R}^3$), do not yield uniqueness. However, since the image of $u$ is not dense in the whole space either (\Cref{prop:cont_uniq}), we also cannot infer uniqueness from this. Thus, a question remains: Is the denseness not only sufficient, but also necessary?
Unfortunately, the answer is no, and the following criteria turn out to be more useful in the context of dynamical systems chaotic on a strange attractor.
\begin{proposition} \label{prop:density-condition-for-uniqueness-of-analytic-functions}
    Let $V$ be the set of analytic functions mapping from $\mathbb{R}^d$ to $\mathbb{R}^d$. Assume there exists $F\in V$ such that $\dot u = F(u)$. Denote with $\mathcal{D}=u(U)$ the image of $u$. Then, $F$ is the unique function in $V$ such that $\dot u = F(u)$ if $\mathcal{D}$ is dense in some analytic set of uniqueness $A\subset \mathbb{R}^d$.   
\end{proposition}
\begin{proof}
 Using the definition of sets of uniqueness and continuity of analytic functions, we get that if a function $G$ vanishes on $\mathcal{D}$, then it also vanishes on $\mathbb{R}^d$. \Cref{prop:general-uniqueness} then directly gives the uniqueness of $F$.    
\end{proof}
\begin{lemma}[Dimension Criterion]\label{lem:haus}
Any set $A$ with $\operatorname{dim}_H A > d-1$ is an analytic set of uniqueness.
\end{lemma}
\begin{proof}
This comes directly from the fact that the Hausdorff dimension of the zero set of a real analytic function is less than or equal to $d-1$. Simple proof of this fact can be found in \citet[Proposition 2]{mityagin} using the implicit function theorem. To finish the proof, we simply use that the Hausdorff dimension of a subset is less than or equal to the Hausdorff dimension of the set.
\end{proof}

\subsubsection{Sets of Uniqueness with Extra Information}

In the subsections above, we established that sets of uniqueness for analytic and continuous function spaces are geometrically quite ``large.'' The requirement for system trajectories to sufficiently sample these large sets is the main limitation for system discovery, which we make concrete in the next section. Here, we instead investigate whether incorporating additional information can reduce the required size of the observation set for analytic functions. 

For example, one may know a global implicit constraint on the function space in the form of a conservation law. Can such information enhance discoverability?  The significance of these results will be demonstrated in \Cref{sec:cons_laws}, enabling discoverability of the harmonic oscillators formally shown in \Cref{cor:uniq_sho_ap}. Our first result is a direct consequence of the implicit function theorem.

\begin{lemma}\label{lem:uniq_w_extra}
Let $F^1, F^2: \Rl^d \to \Rl^d$ be real analytic functions and $H: \Rl^d \times \Rl^d \to \Rl$ be a real analytic function such that $\nabla_vH(F^i(u), u) = 0$ for $i=1,2$ and for all $u \in \Rl^d$. Suppose $F^1$ and $F^2$ agree on a non-empty set $A \subset \Rl^d$ and there exists $p\in A$, such that $\nabla_{v}^2H(F^1(p),p)$ is non-singular. Then $F^1 = F^2$ on all of $\R^d$.
\end{lemma}
\begin{proof}\label{lem:uniq_w_extra_ap}
Let $G: \R^d \times \R^d \to \R^d$ be defined by $G(v, u) = \nabla_v H(v, u)$. Since $H$ is real analytic, so is $G$. By hypothesis, $G(F^i(u), u) = 0$ for $i=1,2,\; u \in \R^d$. We will first utilize the implicit function theorem to show agreement on an open set around $p$, which will then be extended to the whole space using the Identity  \Cref{thm:open-condition-for-uniqueness-of-analytic-functions}.

Let $v_p = F^1(p) = F^2(p)$, as $p\in A$. To use the implicit function theorem, we need to check two conditions. Trivially, 
$G(v_p, p) = 0$. The Jacobian of $G$ with respect to its first variable is $J_v G(v, u) = \nabla_v(\nabla_v H(v, u)) = \nabla_{v}^2 H(v, u)$. By assumption, this matrix is non-singular at $(v_p, p)$. Therefore, there exist open neighborhoods $U \subset \R^d$ of $p$ and $V \subset \R^d$ of $v_p$, and a unique real analytic $\phi: U \to V$ such that for any $(u, v) \in U \times V$, the equation $G(v, u) = 0$ is equivalent to $v = \phi(u)$.

For any $i=1,2$: Since $F^i$ is analytic, it is continuous. Thus, there exists an open neighborhood $U^i \subseteq U$ of $p$ such that for all $u \in U^i$, $F^i(u) \in V$. For any $u \in U^i$, the pair $(F^i(u), u)$ is in $V \times U$ and satisfies $G(F^i(u), u) = 0$. By the uniqueness of $\phi$, it must be that $F^i(u) = \phi(u)$ for all $u \in U^i$. Let $W = U^1 \cap U^2$. $W$ is a non-empty open neighborhood of $p$. For all $u \in W$, we have $F^1(u) = \phi(u) = F^2(u)$. 

With this, we can now use the identity theorem to conclude that $F^1=F^2$ on $\R^d$.
\end{proof}

\begin{nb}
Unfortunately, the result above is only useful whenever the conservation law itself already restricts the number of possible $F$ to a finite set. Existence of a continuous space of distinct possible solutions for $F$ would necessarily make the Hessian singular. Relaxing the invertibility turns out to be impossible in the general case, even when using geometric information of $A$, as illustrated in \Cref{nb:38counter}. This primarily occurs due to the possibility of bifurcations of possible solutions arising from $H(v,u)$. I.e., if $F$ is not an isolated critical point $\nabla_vH$, then we can move by a small distance from $A$, while keeping the critical point non-isolated, meaning that $F_1$ and $F_2$ can continuously separate and become non equal. However, by restricting the space of vector fields, it turns out to be possible to directly utilize geometric information from $A$, complementing directions explored within $A$ with those restricted by $H$.
\end{nb}

\begin{theorem}\label{thm:uniq_symmjac}
Let $F_1, F_2: \R^n \to \R^n$ be two analytic vector fields with symmetric (skew-symmetric) Jacobian, and let $\mathcal{M} \subset \R^n$ be an analytic manifold. Suppose: 
\begin{enumerate}
    \item $F_1(u) = F_2(u)$ for all $u \in \mc{M}$.
    \item There exists an analytic function $H: \R^n \times \R^n \to \R$, s.t., $\nabla_{v} H(F_i(u), u) = 0$, $\forall u \in \R^n$ and $i=1, 2$.
    \item There exists a point $p \in \mc{M}$ such that $\ker(C) \subseteq \TpM$, where $C = \nabla_{v}^2 H(F_1(p), p)$ is the Hessian matrix of $H$ with respect to its first vector argument.
\end{enumerate}
Then, $F_1(u) = F_2(u)$ for all $u \in \R^n$.
\end{theorem}

\begin{proof}
Let $\Delta = F_1 - F_2$. The vector field is analytic with (skew-)symmetric Jacobian. We prove $\Delta \equiv 0$ by showing that $\Delta$ and all its partial derivatives vanish at $p$ inductively. Let $G(v, u) = \nabla_{v} H(v, u)$, then by the Fundamental Theorem of Calculus,
$$ 0 = G(F_1(u), u) - G(F_2(u), u) = \left( \int_0^1 \nabla_v G(F_2(u) + t\Delta(u), u) \, dt \right) \Delta(u) $$
Let $A(u) = \int_0^1 \nabla^2_{v} H(F_2(u) + t\Delta(u), u) \, dt$. This gives us the key equation $A(u)\Delta(u)=0$ for all $u \in \R^n$.

We will now show inductively that for any multi-index $\alpha$, $\partial^\alpha \Delta(p) = 0$ by induction on $k=|\alpha|$. Let $T = d^k \Delta(p)$ denote the $k^{th}$ derivative tensor of $\Delta$ evaluated at $p$. The base case is established trivially, as $\Delta(p) = 0$ for $k=0$. Inductively, we now differentiate $A(u)\Delta(u)=0$ and evaluate at $p$:
\[ 0 = \partial^\alpha(A\Delta)\big|_p = \sum_{\beta \le \alpha} \binom{\alpha}{\beta} (\partial^{\alpha-\beta} A(p)) (\partial^\beta \Delta(p)) = A(p)\partial^\alpha \Delta(p) + \sum_{\beta < \alpha} \binom{\alpha}{\beta} (\partial^{\alpha-\beta} A(p)) (\partial^\beta \Delta(p)) \]
By the inductive hypothesis, $\partial^\beta \Delta(p) = 0$ for all $\beta < \alpha$, so the sum vanishes. We are left with $A(p)\partial^\alpha \Delta(p) = 0$, which is $C(\partial^\alpha \Delta(p)) = 0$, showing that $\im(T)\subseteq\ker(C)$. From Hypothesis 3 of the theorem, we have $\ker(C) \subseteq \TpM$. Chaining these together yields:
$$ \im(T) \subseteq \TpM.$$

As $\Delta \equiv 0$ on $\mc{M}$, $T(v_1,\dots,v_k) = 0$, $\forall v_i\in\TpM$. Without loss of generality, we can now consider an orthonormal basis, such that $\TpM$ is the subspace spanned by the first $m$ basis vectors, $\{e_1, \dots, e_m\}$, where $m$ is the dimension of $\TpM$. In this basis, we will consider elements of the tensor $T^{i_0}_{i_1,\dots,i_k}$. By above, we know that (1) $T^{i_0}_{i_1,\dots,i_k} = 0$, if $i_1,..,i_k \leq m$; (2) $T^{i_0}_{i_1,\dots,i_k} = 0$, if $i_0 > m$ as $\im(T) \subseteq \TpM.$ With these two facts we can now establish that $T$ is zero whenever we consider vector fields with symmetric (skew-symmetric) Jacobians.

By (1), we need to only consider the case when there is at least one $i_n > m$. By (skew-)symmetry, we can swap the indices to give 
$$T^{i_0}_{i_1,\dots,i_n,\dots,i_{k}} = (-1)^sT^{i_n}_{i_1,\dots, i_0,\dots,i_{k}} = 0,$$
by (2) as $i_n > m$. Therefore, all components of the tensor $T$ are zero.
\end{proof}

\begin{nb}\label{nb:38counter}
    Technically, the symmetry condition is only required for the last step. This occurs due to the possibility of nilpotent non-zero matrices, as illustrated in the following. Let $n=2$ and consider the space $\mathbb{R}^2$ with coordinates $u=(x,y)$.
    Let $\mathcal{M} \subset \mathbb{R}^2$ be the $x$-axis, $p = (0, 0)$ and $F_1, F_2: \mathbb{R}^2 \to \mathbb{R}^2$ be the vector fields:
    $$F_1(x,y) = \begin{pmatrix} y \\ 0 \end{pmatrix} \quad \text{and} \quad F_2(x,y) = \begin{pmatrix} \frac12 y^2 \\ 0 \end{pmatrix} \quad \text{with} \quad H(v, u) = \frac{1}{2} v_2^2,$$
    Both fields are analytic and admit the following derivatives:
     $$J_{F_1} = \begin{pmatrix} 0 & 1 \\ 0 & 0 \end{pmatrix} \quad \text{and} \quad J_{F_2} = \begin{pmatrix} 0 & y \\ 0 & 0 \end{pmatrix} \quad \text{with} \quad \nabla_v H(v_1, v_2) = \begin{pmatrix} 0 \\ v_2 \end{pmatrix},$$
    
    Clearly, $F_1(u) = F_2(u) = 0$ for all $u \in \mathcal{M}$, and $\nabla_v H = 0$ for both $F_1,F_2$. By a trivial calculation, $\ker(C) = T_p\mathcal{M} = \operatorname{span}\left\{(1,0)^\top\right\}$. Thus, the conditions of the theorem are satisfied, but clearly the uniqueness is not achieved. That occurs, as the first step of the proof effectively establishes that $\im(T) \subseteq \ker(T)$, or in other words that ${J_{F}}^2 = 0$, which is indeed satisfied in this example. It does not, however, imply that $J$ is zero.
\end{nb}

\begin{nb}
The class of vector fields with symmetric or skew-symmetric Jacobian is not extremely restrictive, as it includes all gradient flow systems, as well as separable Hamiltonian systems.
\end{nb}

\section{Main Results}\label{sec:main}
Having reduced the problem of discoverability to that of sets of uniqueness, we now investigate whether systems trajectories can sufficiently cover such sets, leading to the main results of this paper.
\subsection{The Continuous Case}

The continuous case is again rather simple, as the following can be deduced directly from \Cref{prop:cont_uniq,prop:general-uniqueness}.

\begin{theorem}\label{thm:cont}
Chaotic on the whole domain implies discoverable (from a single trajectory) in the space of continuous functions $C^0$. 
Discoverable in the space of continuous functions from a single trajectory implies topological transitivity.
\end{theorem}

\begin{figure}[t]
    \centering
    \begin{minipage}{0.5\textwidth}
        \includegraphics[width=\linewidth]{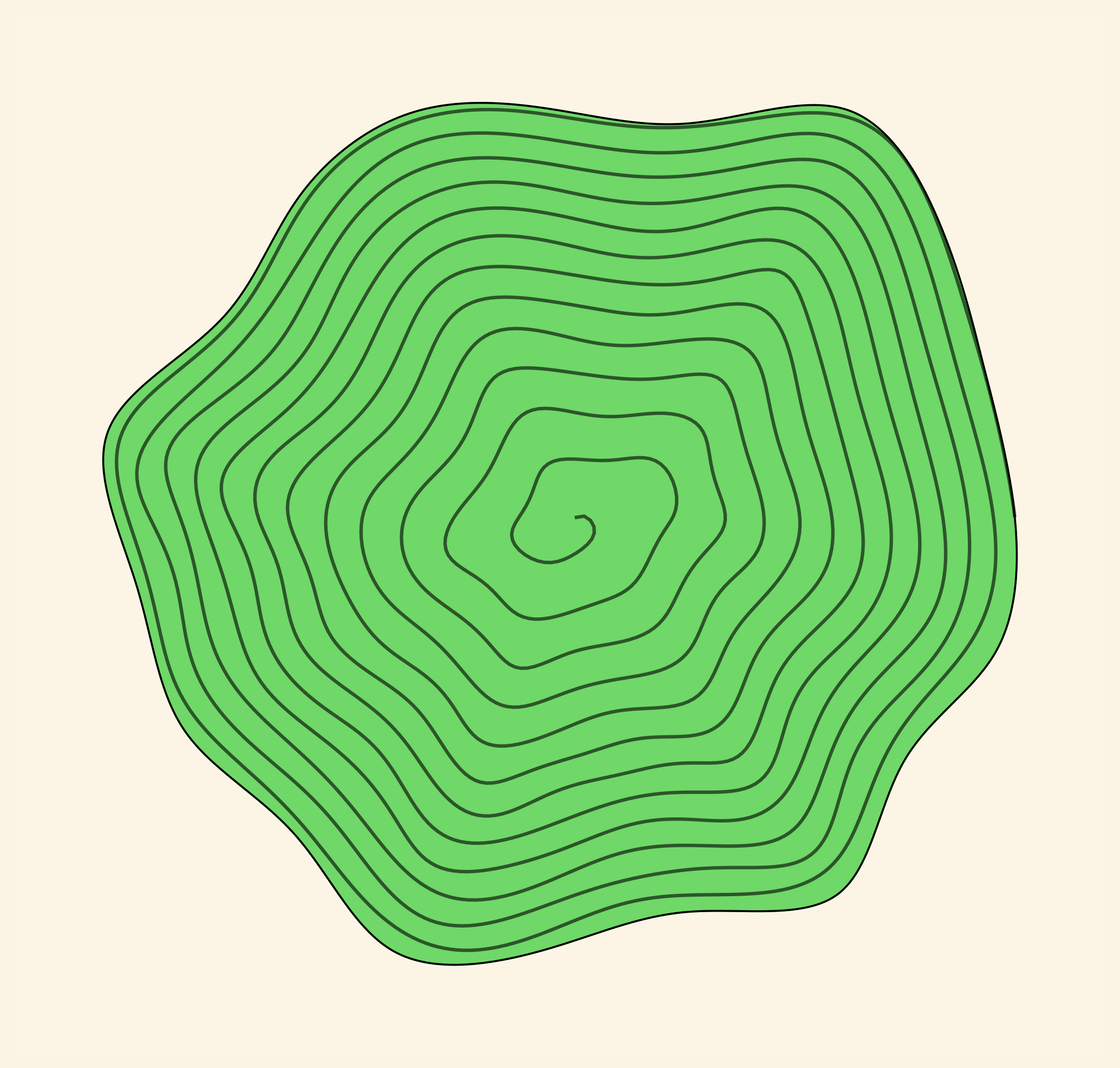}%
    
    \end{minipage}%
    \begin{minipage}{0.5\textwidth}
        \includegraphics[width=\linewidth]{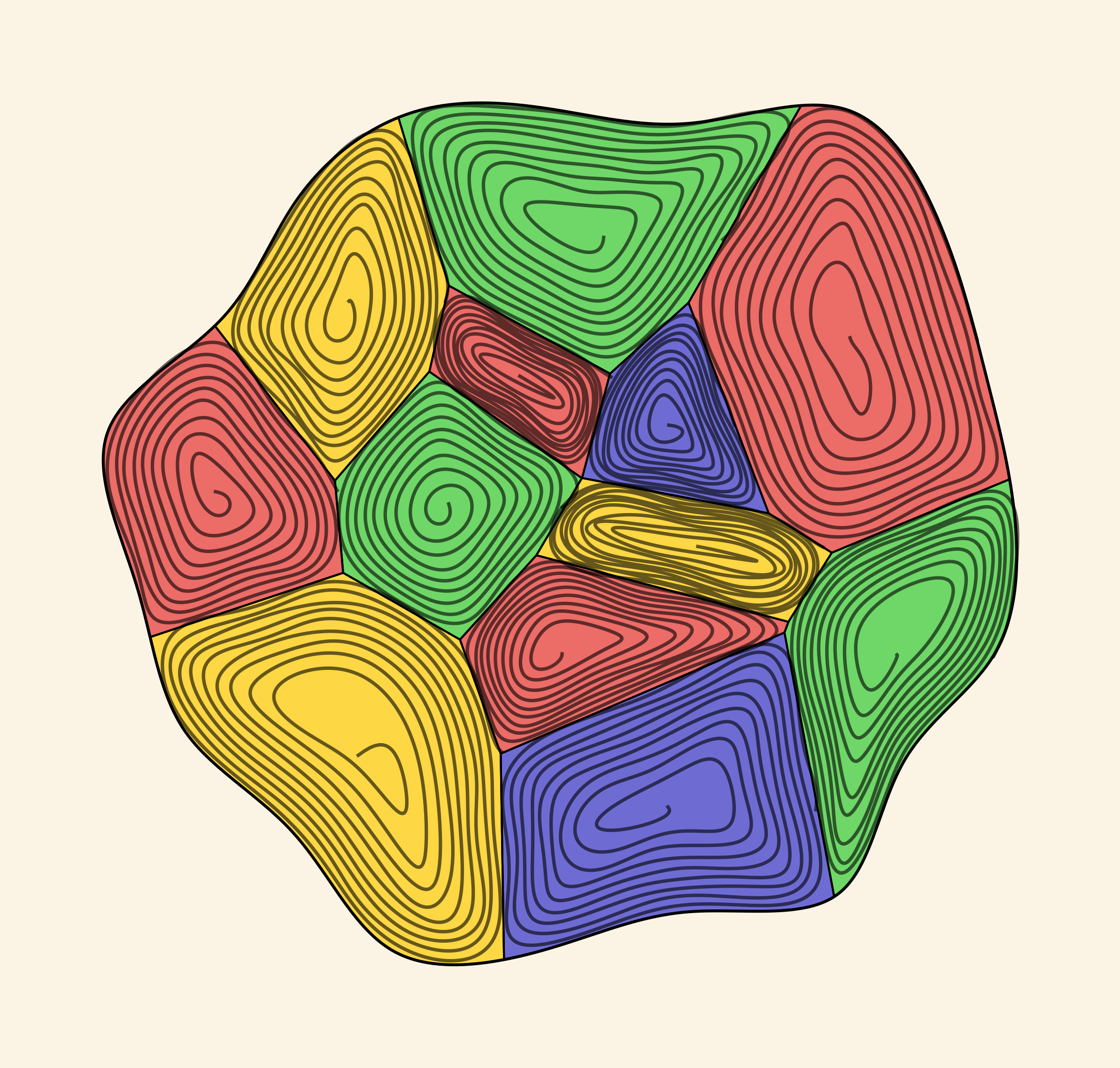}
    \end{minipage}
    \caption{\textbf{Illustration of finite continuous discoverability}. It is equivalent to a decomposition into cells, on each of which the flow is topologically transitive. Indecomposable systems require $1$ trajectory for discovery, while decomposable ones require $1$ for each cell.}
    \label{fig:cells}
\end{figure}

As it is also possible in practice to have access to multiple trajectories, we investigate that case in \Cref{pro:continuous-discoverable-multiple-trajectories}.

\begin{proposition}\label{pro:continuous-discoverable-multiple-trajectories}
    If a system is discoverable from $n$ trajectories through $\{x_1, x_2, \dots, x_n\} \subset X$, then the system is either topologically transitive, and only one of the trajectories $x_i$ is needed for discovery, or $X$ can be expressed as the union of $1 < k \le n$ proper, closed, invariant subsets $\{C_{i_1}, \dots, C_{i_k}\}$, where each $C_{i_j} = \overline{\Orb(x_{i_j})}$. The flow restricted to each component is topologically transitive, and only $k$ trajectories are needed for discovery. Visualized in \Cref{fig:cells}.
\end{proposition}
\begin{proof}
    Here we assume that we are looking at a continuous flow on a compact metric space $X$.
    Discoverable from $n$ trajectories implies that for some $\{x_1, x_2, \dots, x_n\} \subset X$, we must have by \Cref{prop:general-uniqueness,prop:cont_uniq} that the union of the closures of their trajectories covers the space:
    $$X = \bigcup_{i=1}^{n} \overline{\Orb(x_i)} .$$

    Now, if the systems is topologically transitive, then the result of \citet[Exercise 6.2.5]{grosse2011linear} and proof in \citet[Theorem 6.6]{grosse2011linear} can be combined to show that one of $\Orb{x_i}$ must be dense in $X$. Note, that results of \citet{grosse2011linear} are not directly applicable, as those are established for linear operators, and in general do not extend to non-linear ones. However, under the assumption of topological transitivity these can similarly be established. And thus, $\Orb{x_i}$ is sufficient for discovery.
    
    We will now assume that the flow is not topologically transitive. We know that each $C_{i} = \overline{\Orb(x_{i})}$ is closed and $\phi$-invariant. To see this, let $y \in C_i$. Then there exists a sequence $\{t_k\}_{k \in \mathbb{N}} \subset \mathbb{R}$ such that $\Phi_{t_k}(x_i) \to y$. For any $s \in \mathbb{R}$, the continuity of the flow map $\Phi_s$ implies that $\Phi_s(y) = \Phi_s(\lim_{k\to\infty} \Phi_{t_k}(x_i)) = \lim_{k\to\infty} \Phi_{s+t_k}(x_i)$. Since each point $\Phi_{s+t_k}(x_i)$ is in the trajectory $\Orb(x_i)$, their limit point $\Phi_s(y)$ must be in its closure, $C_i$. Thus, $\Phi_s(C_i) \subset C_i$ for all $s \in \mathbb{R}$.
    Each set $C_i$ must be a {proper} subset of $X$. If, for some index $j$, we had $C_j = X$, then by definition, the trajectory of $x_j$ would be dense in $X$. By Birkhoff's \Cref{lem:birk}, this would imply topological transitivity. This would contradict the premise that the system is not topologically transitive. Therefore, $C_i \neq X$ for all $i \in \{1, \dots, n\}$. Thus, the collection of sets $\{C_1, \dots, C_n\}$ (or a minimal sub-collection thereof whose union is $X$) constitutes a decomposition of $X$ into proper, closed, and invariant subsets. And thus the $k$ trajectories are sufficient for discovery.

    By construction, the trajectory of the point $x_j$, $\Orb(x_j)$, is dense in the space $C_j$, and the space $C_j = \overline{\Orb(x_j)}$ is a compact metric space in its own right. Applying now Birkhoff's \Cref{lem:birk}, restricted to $C_j$, we have that $(\phi, C_j)$ is topologically transitive for each $j \in \{1, \dots, n\}$.
\end{proof}

\subsection{The Analytic Case}
The analytic case turns out to be much more difficult than the continuous case, in the same way as in \Cref{sec:sets_of_uniqueness}. In order to establish the impossibility of discoverability, we show that there exists a natural connection between chaos and analytical geometry: all discoverable trajectories must fall into the attracting set. 

\subsubsection{Analytic Trajectories} \label{app:trajectories}

A general condition for this is extremely difficult, without a greater understanding of real analytic singularities mirroring the case of complex analytic geometry. Indeed, the following example shows that even semianalytic sets are not necessarily contained in some analytic hypersurface.

\begin{example}
    Consider the map $(0,\infty) \rightarrow \mathbb{R}^2$ given by
    \begin{equation*}
        x \mapsto \left(x, \sin\left(\frac{1}{x+1}\right)\right)
    \end{equation*}
    This is analytic, as it avoids the singularity at $x = -1$. 

    Letting $\gamma$ be the image, we see that it is subanalytic. Indeed, it is even semianalytic which can be seen in the following way. Away from $x = -1$, it is given as
    \begin{equation*}
        \gamma \cap \{x \neq -1\} = \{y - \sin\left(\frac{1}{x+1}\right) = 0, 0 \leq x \}
    \end{equation*}
    Around the line $x = -1$ we may take a neighborhood avoiding $\gamma$ and thus is trivially semianalytic near $x = -1$. We note that $\gamma$ is not C-semianalytic.

    We now claim that even though $\gamma$ is semianalytic, it is a set of uniqueness. Suppose not, and let $G$ be a global analytic function such that $\gamma \subseteq \{G = 0\}$. Then we define
    \begin{equation*}
        h(x) = G\left(x, \sin\left(\frac{1}{x+1}\right)\right) = 0, \quad
    \end{equation*}
    This is analytic on $(-1, \infty)$ and by assumption vanishes on $(0,\infty)$. Hence it vanishes on all of $(-1,\infty)$. Taking the limit as $x \rightarrow -1^+$, the function $\sin \frac{1}{x+1}$ accumulates to all values in the interval $[-1,1]$. Hence by continuity, we see that $G(-1,y) = 0$ for all $y \in [-1,1]$. Again, this implies $G(-1, y) = 0$ for all $y$. Hence we may write
    \begin{equation*}
        G(x,y) = (x + 1)G^1(x,y)
    \end{equation*}
    for some global analytic function $G^1$. As $x+1$ does not vanish on $\gamma$, we must have that $G^1$ vanishes on $\gamma$. Repeating this argument for $G^1$, we get that $G$ is divisible by $(x+1)^N$ for any $N$, but this is impossible.
\end{example}

The following result shows that this example occurs due to it being simply semianalytic and not C-semianalytic. 

\begin{lemma}
    Suppose $Z$ is a compact C-semianalytic set with vanishing interior. Then $Z$ is not a set of uniqueness.
\end{lemma}
\begin{proof}
    We may find finitely many open sets $U_i$ covering $Z$ such that on each open we have global analytic function $f_{i,1},\dots,f_{i,n_i},g_{i,1},\dots,g_{i,m_i}$ with
    \begin{equation*}
        Z \cap U = \{y \in U \mid f_{i,j}(y) = 0, g_{i,k}(y) > 0 \quad \forall j,k\}
    \end{equation*}
    Define $F_i = \sum_{j} f_{i,j}^2$. Then $F_i$ is a globally analytic function vanishing on $Z \cap U$. Finally, we define $F = \prod_{i} F_i$. This is then a globally analytic function vanishing on $Z$. As long as we can show that this is non-zero we will be done. Assume it is zero. Then one of the $F_i$ is zero, and hence all the $f_{i,1}, \dots, f_{i,n_i}$ are necessarily zero. But then $Z$ will contain the non-empty open set $\{g_{i,k} > 0\}$, contradicting $Z$ having vanishing interior. Thus $F$ is not the zero function.
\end{proof}

We now look at conditions for when a trajectory will be a set of uniqueness. We shall also assume that $\phi$ is maximally analytically continued. By this we mean there is no open set $D' \subseteq \mathbb{R}$ and analytic function $\widetilde{\phi} : D' \rightarrow \mathbb{R}^d$ such that $D \subsetneq D'$ and $\widetilde{\phi}|_{D} = \phi$. By the following result, we have that this does not change the answer.

\begin{proposition}\label{prop:ReductionToGraphs}
    Suppose $\widetilde{\phi} : \widetilde{D} \rightarrow \mathbb{R}^d$ is an analytic continuation of $\phi : D \rightarrow \mathbb{R}^d$. Then $\gamma = \operatorname{img}(\phi)$ is a set of uniqueness if and only if $\widetilde{\gamma} = \operatorname{img}(\widetilde{\phi})$ is.
\end{proposition}
\begin{proof}
    One direction is clear. For the other, let $F$ be a global real analytic function such that $\gamma \subseteq \{F = 0\}$. Then we define $\widetilde{h} = F \circ \widetilde{\phi}$. This is an analytic function on $\widetilde{D}$, and because $\widetilde{\phi}$ agree with $\phi$ on $D$, we have that $\widetilde{h}$ vanishes on $D$. Therefore, $\widetilde{h} \equiv 0$. 
\end{proof}

We may also complexify $\phi$. The power series defining $\phi$ will have some radius of convergence, such that this extends to a holomorphic function on some open $D \subseteq \widetilde{D} \subseteq \mathbb{C}$. We may chose $\widetilde{D}$ such that it is invariant under complex conjugation. 

\begin{proposition}\label{prop:ExistsProperExtension}
    Suppose the holomorphic extension $\widetilde{D} \rightarrow \mathbb{C}^d$ is proper. Then $\gamma$ is not a set of uniqueness.
\end{proposition}
\begin{proof}
    We use \citep[Theorem 1.1]{fernando2015irreduciblecomponentsgloballydefined}. This gives that the image of $D$ is a tame C-semianalytic set inside of $\mathbb{R}^d$ of dimension $1$. Then by Proposition 5.12 loc. cit. we obtain that the smallest analytic subset containing $\gamma$ is also $1$ dimensional, as required.
\end{proof}

Checking when this proper holomorphic extension exists is unfortunately subtle. We have some results in this direction.

\begin{corollary}
    Suppose that each coordinate of $\phi$ is a rational function. Then $\gamma$ is not a set of uniqueness. In particular, if each coordinate of $\phi$ is a polynomial.
\end{corollary}
\begin{proof}
    We may extend $\phi$ to a proper map $\mathbb{C}\mathbb{P}^1 \backslash P \rightarrow \mathbb{C}^d$. Here $P$ are the poles of all the coordinate functions, which is a discrete set of points. Also $P$ is non-empty as otherwise each coordinate is constant, and then $\phi$ would be constant, which would contradict our standing assumption. Also note that $U$ is invariant under complex conjugation as each coordinate of $\phi$ is real-valued. Therefore $\mathbb{C}\mathbb{P}^1 \backslash P$ is an invariant Stein-neighborhood of $\mathbb{R}\mathbb{P}^1$. Then we apply \Cref{prop:ExistsProperExtension}.
\end{proof}

\subsubsection{Trajectories from Flows}

We fix a dynamical system, with $F$ an analytic function.
\begin{equation*}
    \frac{\partial u}{\partial t} = F(u)
\end{equation*}
Let $\phi$ be the corresponding flow. 
Let $D$ be the maximal open set where this flow is defined. Let $\gamma$ be the trajectory starting from $x_0$. 

In this situation we have greater control over the trajectory through using the defining equation. First we define the Lie derivative with respect to the vector field $F$. For an analytic function $g$ we let
\begin{equation*}
    L_F(g) = \nabla g \cdot F
\end{equation*}
It is standard that this is analytic. This is linear and satisfies the Leibniz rule for products
\begin{equation*}
    L_F(gh) = hL_F(g) + gL_F(h)
\end{equation*}

\begin{definition}
    The ideal of functions vanishing on $\gamma$ is defined as
    \begin{equation*}
        I_\gamma = \{g \in C^{\omega}(\mathbb{R}^d) \mid g|_{\gamma} \equiv 0 \} \subseteq C^{\omega}(\mathbb{R}^d)
    \end{equation*}
\end{definition}

Notice that discoverability is exactly just asking that this ideal be simply the zero function. Therefore we shall focus on the properties of this ideal, and how one might find a non-zero element.

\begin{proposition}\label{prop:IdealIsPrime}
    $I_{\gamma}$ is prime.
\end{proposition}
\begin{proof}
    Suppose that $hg \in I_{\gamma}$. Then we have
    \begin{equation*}
        h(\phi^t(x_0))g(\phi^t(x_0)) = 0 \quad \forall t \in [0, \infty)
    \end{equation*}
    The ring of analytic functions is an integral domain, hence one of these must be the zero function. Therefore one of them is contained in $I_{\gamma}$ and thus $I_{\gamma}$ is prime.
\end{proof}

\begin{proposition}
    $I_\gamma$ is an $L_F$-invariant ideal.
\end{proposition}
\begin{proof}
    Suppose that $g \in I_\gamma$. Consider the function $h(t) = g(\phi^t(x_0))$. Then by the chain rule
    \begin{equation*}
        h'(t) = \nabla g(\phi^t(x_0)) \cdot F(\phi^t(x_0))
    \end{equation*}
    Therefore, $h' = L_F(g)$. But we know $h$ is just the zero function, hence $h'$ is the zero function. Therefore $L_F(g) \in I_\gamma$.
\end{proof}

\begin{proposition}
    $g \in I_{\gamma}$ if and only if $(L_F)^k(g)(x_0) = 0$ for all $k \geq 0$. If 
\end{proposition}
\begin{proof}
    One direction is clear from the previous proposition. Now assume that $(L_F)^k(g)(x_0) = 0$ for all $k$. Consider again the function $h$ from the previous argument. The $(L_F)^k(g)(x_0)$ give the coefficients of the Taylor series of $h$ at $0$, hence these coefficients are all zero. Thus $h$, being analytic, is identically zero. 
\end{proof}

\begin{definition}
    Given an analytic function $g \in C^{\omega}(\mathbb{R}^d)$ we define the cofactor by
    \begin{equation*}
        K = L_F(g) / g
    \end{equation*}
    This is a meromorphic function on $\mathbb{R}^d$.
\end{definition}

\begin{definition}
\begin{enumerate}
    \item A flow semi-invariant for $\gamma$ is a non-zero analytic function $g \in C^{\omega}(\mathbb{R}^d)$ such that $L_F(g) = hg$ for some $h \in C^{\omega}(\mathbb{R}^d)$. 

    \item A flow semi-invariant vector for $\gamma$ is a non-zero vector of analytic functions $(g_i) \in C^{\omega}(\mathbb{R}^d)^k$ such that there exists a matrix of analytic functions $A = (a_{ij})$ such that
    \begin{equation*}
        L_F \begin{pmatrix}
            g_1 \\ g_2 \\ \vdots \\ g_k
        \end{pmatrix} = A\begin{pmatrix}
            g_1 \\ g_2 \\ \vdots \\ g_k
        \end{pmatrix} 
    \end{equation*}

    \item A local flow semi-invariant for $\gamma$ is a non-zero analytic function $g \in C^{\omega}(\mathbb{R}^d)$ such that there is some neighborhood $U$ of $x_0$ and an analytic function $K \in C^{\omega}(U)$ such that $L_F(g) = Kg$ on $U$.

    \item A function with locally bounded cofactor is a non-zero analytic function $g \in C^{\omega}(\mathbb{R}^d)$ such that there is a constant $C > 0$ such that in a neighborhood of $x_0$ we have
    \begin{equation*}
        |L_F(g)| \leq C |g - g(x_0)|
    \end{equation*}

    \item A flow invariant ideal for $\gamma$ is a non-zero $L_F$-invariant ideal $J \subseteq C^{\omega}(\mathbb{R}^d)$ such that for all $f \in J$ we have $f(x_0) = 0$.

\end{enumerate}

    When the trajectory $\gamma$ is clear, we shall often drop it from the notation.
\end{definition}

In general, every function $g$ has a meromorphic cofactor $K = L_F(g)/g$. Each of these definitions is imposing a certain level of regularity on this cofactor. We now show that any of these is sufficient to imply non-uniqueness.

\begin{proposition}\label{prop:SemiInvariantVectorImpliesNonUniqueness}
    If $(g_i)$ is a flow semi-invariant vector for $\gamma$ then for each $i$ we have $g_i - g_i(x_0) \in I_{\gamma}$.
\end{proposition}
\begin{proof}
    We have $L_F(g_i) = \sum_j a_{ij}g_j$. Therefore $L_F(g_i)(x_0) = \sum_j a_{ij}(x_0)g_j(x_0) = 0$. Inductively now assume that we have shown $(L_F)^{k-1}(g_i)(x_0) = 0$ for each $i$. By the Liebniz rule we have
    \begin{equation*}
        (L_F)^k(g_i)(x_0) = (L_F)^{k-1}(\sum_j a_{ij} g_j)(x_0) = \sum_j a_{ij}(x_0)(L_F)^{k-1}(g_j)(x_0) = 0
    \end{equation*}
    Hence we obtain $g_i \in I_{\gamma}$ for all $i$.
\end{proof}

\begin{proposition}
    If $g$ is a local flow semi-invariant for $\gamma$ then $g - g(x_0) \in I_{\gamma}$.
\end{proposition}
\begin{proof}
    We have that there is a neighborhood $U$ of $x_0$ on which there is an analytic function $K \in C^{\omega}(U)$ such that
    \begin{equation*}
        L_F(g) = Kg
    \end{equation*}
    Now apply the argument of \Cref{prop:SemiInvariantVectorImpliesNonUniqueness} in the open $U$. This gives that $g$ vanishes on $\gamma$ inside of $U$. But then by the analytic identity theorem, it vanishes on all of $\gamma$. Indeed, writing $\phi : D \rightarrow \mathbb{R}^d$ for the flow map, then $g \circ \phi$ vanishes on $g^{-1}(U)$ which is a non-empty open set of $D$. Hence $g \circ \phi$ vanishes on $D$, hence $g$ vanishes on $\gamma$.
\end{proof}

\begin{proposition}
    If $g$ is a function with locally bounded cofactor then $g - g(x_0) \in I_{\gamma}$.
\end{proposition}
\begin{proof}
    For this we consider the function $h(t) = (g - g(x_0)) \circ \phi^t(x_0)$. Then $h'(t) = (L_Fg) \circ \phi^t(x_0)$. Thus in some neighborhood of $0$ we have the following bound
    \begin{equation*}
        |h'(t)| \leq C |h(t)|
    \end{equation*}
    Choose the neighborhood of $0$ small enough such that both $h'$ and $h$ do not change signs. Then we apply Gronwall's inequality to obtain
    \begin{equation*}
        |h(t)| \leq |h(0)| e^{tC} = 0
    \end{equation*}
    Therefore $h$ vanishes on some neighborhood of $0$ and thus on the entire domain. This implies $g \in I_{\gamma}$.
\end{proof}

\begin{proposition}
    If $J$ is a flow invariant ideal for $\gamma$ then we have $J \subseteq I_{\gamma}$.
\end{proposition}
\begin{proof}
    By induction on $L_F(J) \subseteq J$, we obtain $(L_F)^k(J) \subseteq J$. Therefore, as $x_0 \in V(J)$, we have $(L_F)^k(g)(x_0) = 0$ for all $k \geq 0$. Thus $g \in I_{\gamma}$.
\end{proof}

If $I_{\gamma}$ is non-zero there is always a trivial flow-invariant ideal, namely given by $J = I_{\gamma}$, so these are not too interesting. We therefore investigate the necessity of flow semi-invariants and flow semi-invariant vectors more generally.

First we require a technical result allowing us to rephrase the local flow semi-invariant condition.

\begin{proposition}
    $\gamma$ is a set of uniqueness if and only if there are no functions with locally bounded cofactor.
\end{proposition}
\begin{proof}
    One direction is exactly the previous results. Therefore we deal with the converse. Suppose that $\gamma$ is not a set of uniqueness. Recall $\gamma$ is a set of uniqueness if and only if $I_{\gamma} = 0$. Therefore we have $I_{\gamma} \neq 0$. 

    Choose $g \in I_{\gamma}$. We then obtain a sequence of ideals, by repeatedly applying $L_F$ to $g$ as follows
    \begin{equation*}
        (g) \subseteq (g, L_Fg) \subseteq (g, L_Fg, (L_F)^2g) \subseteq \dots \subseteq C^{\omega}(\mathbb{R}^d)
    \end{equation*}
    We now take the germs of all of these functions at $x_0$ to obtain the sequence of ideals
    \begin{equation*}
        (g_{x_0}) \subseteq (g_{x_0}, (L_Fg)_{x_0}) \subseteq (g_{x_0}, (L_Fg)_{x_0}, ((L_F)^2g)_{x_0}) \subseteq \dots \subseteq \mathcal{O}_{\mathbb{R}^d,x_0}
    \end{equation*}
    Here $\mathcal{O}_{\mathbb{R}^d,x_0}$ is the local ring of analytic function germs at $x_0$. By \Cref{thm:WeierstrassPreperation} this ring is Noetherian. Hence this infinite sequence of ideals stabilizes, therefore for some $k$ we have
    \begin{equation*}
        ((L_F)^{k+1}g)_{x_0} \in (g_{x_0}, (L_Fg)_{x_0}, \dots, ((L_F)^kg)_{x_0})
    \end{equation*}
    Therefore there are analytic function germs $a_{i, x_0}$ such that
    \begin{equation*}
        ((L_F)^{k+1}g)_{x_0} = \sum_i a_{i,x_0} ((L_F)^ig)_{x_0}
    \end{equation*}
    Choose a neighborhood of $x_0$ such that all the germs $a_{i,x_0}$ lift to analytic functions on the neighborhood. Then we have that for $G = (g, L_Fg, \dots, (L_F)^kg)^T$ locally on this neighborhood
    \begin{equation*}
        L_FG = AG
    \end{equation*}
    for some matrix of analytic function $A$. Define now
    \begin{equation*}
        K = g^2 + (L_Fg)^2 + \dots + ((L_F)^kg)^2
    \end{equation*}
    We claim that this is a flow bounded function. Note that it is globally analytic because $g$ is. Further because $g \in I_{\gamma}$, we have $(L_F)^ig(x_0) = 0$ so $K(x_0) = 0$. We must therefore check that $L_FK$ is bounded by $K$ locally around $x_0$. We compute that locally,
    \begin{equation*}
        L_FK = 2G^TAG
    \end{equation*}
    The matrix $A$ is bounded after possibly shrinking the neighborhood around $x_0$. Therefore we obtain
    \begin{equation*}
        |L_FK| \leq C ||G^2|| = C |K|
    \end{equation*}
\end{proof}

\begin{remark}
    Notice that $g$ is a local flow semi-invariant if and only if the germ $g_{x_0}$ is an eigenvector of $L_F$ acting on the local ring $\mathcal{O}_{\mathbb{R}^d,x_0}$. Indeed suppose we have
    \begin{equation*}
        L_F(g_{x_0}) = k g_{x_0}
    \end{equation*}
    for some $k \in \mc{O}_{\mathbb{R}^d, x_0}$. There is some neighborhood $U$ around $x_0$ where $k$ is the germ of an analytic function $K$ on $U$. This is then exactly the definition of a local flow semi-invariant.

    Let $\mathfrak{p}_{x_0} \subseteq \mathcal{O}_{\mathbb{R}^d, x_0}$ be the maximal ideal of function germs which vanish at $x_0$. Then we have reduced the question of discoverability to the following algebraic question:

    Does the linear operator $L_F : \mathcal{O}_{\mathbb{R}^d, x_0} \rightarrow \mathcal{O}_{\mathbb{R}^d, x_0}$ have any eigenvectors contained in $\mathfrak{p}_{x_0}$?
\end{remark}

To finish this part, we are able to use the theory of semi-integral closure to give a completely algebraic characterization of discoverability. First we introduce the relevant definition.

\begin{definition}
    Suppose that $I \subseteq \mathcal{O}_{\mathbb{R}^d, x_0}$ is an ideal. The real integral closure $\overline{I}^{\mathbb{R}}$ is defined as the set of all $h \in \mathcal{O}_{\mathbb{R}^d}$ such that for all analytic maps $\psi : (\mathbb{R}, 0) \rightarrow (\mathbb{R}^d, x_0)$ we have $h \circ \psi \in (\psi^*I)\mathcal{O}_{\mathbb{R}, 0}$.
\end{definition}

There is a completely algebraic definition due to Roberts in \citet[Proposition 8]{brumfiel2006real}. The most useful characterization for us is the following \citep[Proposition 4.2]{integralclosureofmodules}

\begin{proposition}
    $h \in \overline{I}^{\mathbb{R}}$ if and only if for any choice of generators $(g_i)$ of $I$ there exists a neighborhood $U$ of $x_0$ such that $|h(y)| \leq C \sup_i |g_i(y)|$ for all $y \in U$.
\end{proposition}

\begin{corollary}
    $g$ non-zero has locally bounded cofactor if and only if $L_F(g) \in \overline{(g - g(x_0))}^{\mathbb{R}}$
\end{corollary}

\begin{proposition}\label{prop:PeriodicImpliesImmersion}
    $\gamma$ does not have any self-intersections.
\end{proposition}
\begin{proof}
    Any self-intersections of $\gamma$ will come from non-injectivity of the flow starting at $x_0$. Thus suppose there are two times $t_0, t_1$ such that
    \begin{equation*}
        \phi^{t_0}(x_0) = \phi^{t_1}(x_0)
    \end{equation*}
    Now these are both analytic functions satisfying the same differential equation, hence have the same derivative. Therefore, for any $t > 0$ we have
    \begin{equation*}
        \phi^{t_0 + t}(x_0) = \phi^{t_1 + t}(x_0)
    \end{equation*}
    Thus necessarily we get global periodicity, $\phi^{t}(x_0) = \phi^{t + t_1 - t_0}(x_0)$ for any $t$ such that $t + t_1 - t_0$. 

    Choose $T$ the minimal period for this flow. Then we may restrict $\phi$ to $[0,T]$ to obtain $\widetilde{\phi}$ without changing the image $\gamma$. Now $\widetilde{\phi}$ is injective except at the ends. Indeed if it is not, then by the argument above we obtain a smaller period, which contradicts the minimality of $T$. This then implies that $\gamma$ has no self-intersections.
\end{proof}

\begin{proposition}
    If $\gamma$ is periodic, then it is not a set of uniqueness.
\end{proposition}
\begin{proof}
    Let $T > 0$ be the minimal period. Then restricting $\phi$ to $[0,T]$ we obtain a parameterization of $\gamma$ which we shall call $\widetilde{\phi}$. By \Cref{prop:PeriodicImpliesImmersion} we obtain that $\widetilde{\phi}$ is a real analytic embedding from a compact set. Therefore $\gamma$ is an properly embedded real-analytic submanifold of $\mathbb{R}^d$. Then we are done by \citet[Corollary 7.17]{CieliebakEliashberg2012}, by taking any of the real analytic functions constructed in loc. cit.
\end{proof}

\begin{definition}
    The $\omega$-limit set of $\gamma$ is defined as
    \begin{equation*}
        \lim_{\omega} \gamma = \bigcap_{s \in [0, +\infty)} \overline{\{\gamma(t) \mid t > s\}}
    \end{equation*}

    If $\gamma \cap \lim_\omega \gamma = \varnothing$ then $\lim_\omega \gamma$ is known as a limit cycle.
\end{definition}

The precise relationship between the attractor and uniqueness is subtle, as it doesn't depend just on $\lim_{\omega} \gamma$ but on the exact way that $\gamma$ approaches its limit cycle locally. This is illustrated by the following simple example of a system which traces out converging spirals.

\begin{example}\label{ex:NonChaoticDiscoverableSystem}
    Consider the system given on $\mathbb{R}^2$ which is given as follows, with $\epsilon < 0$ a small constant
    \begin{equation*}
        \frac{du}{dt} = \begin{pmatrix}
            -u_2 + \epsilon u_1 \\ u_1 + \epsilon u_2
        \end{pmatrix}
    \end{equation*}
    This is analytically solvable, and has solution
    \begin{align*}
        u_1(t) &= e^{\epsilon t}(A \cos t - B \sin t) \\
        u_2(t) &= e^{\epsilon t}(A \sin t + B \cos t)
    \end{align*}
    Here $A, B$ are constants depending on the initial conditions. Then the trajectories of this system spiral inwards infinitely, and all trajectories have $N = \{0 \}$. Indeed, taking $t_n = 2\pi n$ we have
    \begin{equation*}
        u(t_n) = (e^{2\epsilon \pi n} A,  e^{2 \epsilon \pi n} B)
    \end{equation*}
    Taking $n \rightarrow \infty$ and noting that $\epsilon < 0$, one sees that this limit is zero, but zero is a critical point for the system, meaning it lies on its own trajectory. 

    It is clear that this system has no dense trajectories, but we shall now show that the system is uniquely discoverable from any trajectory. By rescaling, we may assume that we take the trajectory which starts at $u = (1, 0)$, which corresponds to a choice of constants $A = 1, B = 0$. Thus the trajectory is given by
    \begin{equation*}
        u(t) = (e^{\epsilon t} \cos t, e^{\epsilon t} \sin t)
    \end{equation*}
    Let $\gamma = \operatorname{img} u$. Suppose $f : \mathbb{R}^2 \rightarrow \mathbb{R}$ is real-analytic and satisfies $f|_{\gamma} = 0$. Consider a straight line through the origin of $\mathbb{R}^2$, which we call $\ell$. Then $f|_\ell$ is a one-dimensional analytic function, which necessarily has zeroes on $\ell \cap \gamma$. But $\ell \cap \gamma$ has an accumulation point at $0$. Hence $f|_\ell = 0$. As this holds for any straight line through the origin, we see that $f = 0$. Hence the system is uniquely discoverable from a single trajectory, even though its trajectories are not dense. 
\end{example}

With this, we can finally collect the main theorems together:
\begin{theorem}\label{thm:HausdorffCodim2Thm}
    Suppose that $\Omega$ is a generalized limit cycle which has Hausdorff dimension $\operatorname{dim}_H \Omega > d-2$. Then almost every trajectory $\gamma$ with $\lim_{\omega} \gamma =\Omega$ is discoverable. In particular, for a chaotic system with attractor $\Lambda$ satisfying $\operatorname{dim}_H \Lambda > d-2$ then almost every trajectory attracting to $\Lambda$ is discoverable.
\end{theorem}
\begin{proof}
    We first claim $I_{\Omega}$ is principal. Suppose it was not and there were two functions $(f,g) \subseteq I_{\Omega}$. We may assume that $f,g$ are both irreducible, as $I_{\Omega}$ is prime by \Cref{prop:IdealIsPrime}. Inducting the Hausdorff dimension argument from before, we have that $\dim_H V(f,g) \leq d-2$ as these do not share a common factor. But also by $(f,g) \subseteq I_{\Omega}$ we have $\Omega \subseteq V(f,g)$. This contradicts the Hausdorff dimension assumption on $\Omega$. Therefore $I_{\Omega} = (f)$ for some irreducible $f$. Now assume that $\gamma$ approaches $\Omega$. Again by \Cref{prop:IdealIsPrime} we have that $I_{\gamma}$ is prime. We also have $I_{\gamma} \subseteq I_{\Omega} = (f)$. A prime principal ideal is height one, so the only options for $I_{\gamma}$ are $0$ and $(f)$. Thus we have that $\gamma$ is either contained in ${f = 0}$ or discoverable. But the set of those trajectories contained in ${f = 0}$ has Hausdorff dimension $d - 1$, and thus measure zero inside of the basin of attraction.
\end{proof}

\begin{theorem}\label{thm:main_anal}
If there exists a generalized limit cycle $\Omega$ which is a set of uniqueness then the system is discoverable from any trajectory $\gamma$ such that $\lim_{\omega} \gamma = \overline{\Omega}$. For a chaotic system we in particular have that if an attractor is a set of uniqueness (e.g. if the attractor has Hausdorff dimension greater than $d - 1$), then the system is discoverable.

\end{theorem}

\begin{proof}
    Suppose $\Omega$ is a generalized limit cycle which is a set of uniqueness. Let $\gamma$ be a trajectory such that $\lim_{\omega} \gamma$. Then we know by definition $\Omega \subseteq \overline{\gamma}$. Then if there was a non-zero analytic function $F$ such that $\gamma \subseteq F^{-1}(0)$, then we would necessarily have $\Omega \subseteq F^{-1}(0)$ as well. But this would contradict $\Omega$ being a set of uniqueness. Hence $\gamma$ must be a set of uniqueness. 

    The specialization of this statement to chaotic systems is a simply renaming of terms, as generalized limit cycles will be the attractors in the chaotic case. The statement on Hausdorff dimensions follows then from \Cref{lem:haus} as the zero-set of a non-zero analytic function has Hausdorff dimension at most $d-1$.
\end{proof}

\begin{corollary}\label{cor:lorenz}
    The classical Lorenz attractor is analytically discoverable. 
\end{corollary}
\begin{proof}
There exist many computational tools to estimate the Hausdorff dimension of the Lorenz attractor, estimating it to be $2.06\pm0.01$, and thus within the applicability range of \Cref{thm:main_anal}. However, analytical results are quite scarce. Generally speaking, proving anything directly for the Lorenz system is a rather complicated task, and most of the literature instead focuses on the so-called geometric Lorenz flow \citep{araujo2010three, galatolo2009lorenzlikeflowsexponential}. Luckily, for the geometric Lorenz flow, it was recently shown that the Hausdorff dimension is strictly larger than $2$ \citep{moreira2020hausdorff}.

With this, we simply need to establish the same for the original Lorenz attractor, which holds as \citet{tucker1999lorenz,tucker2002rigorous} proved that the geometric model is valid, so the Lorenz equations define a geometric Lorenz
flow. See e.g. discussion in \citet{luzzatto2005lorenz,pinsky2023analytical}. Thus, we can apply \Cref{thm:main_anal}, resulting in discoverability.
\end{proof}
\begin{nb}\label{nb:chaos}
    \Cref{thm:main_anal} suggests that many known chaotic systems, like the Rossler system \citep{rossler1976equation} ($D \approx 2.01$ \citep{froehling1981determining}), multi-scroll system (or hyperchaotic Chua's circuit) \citep{chua2003double} ($D \approx 2.13$ \citep{1085791}), Chen and L\"u systems \citep{lu2002new} ($D \approx 2.18$ and $D \approx 2.06$ \citep{chen2013nonequivalence,leonov2015estimation}), hyperchaotic Rossler \citep{rossler1979equation} ($D \approx 3.006$ \citep{froehling1981determining}), Rabinovich–Fabrikant system \citep{rabinovich1979stochastic} ($D \approx 2.19$ \citep{grassberger1983measuring}) are in fact discoverable. However, establishing such a fact rigorously is difficult, as it is necessary to both establish existence of an attractor, and provide exact lower bounds on the dimension of it. As noted by \citet{moreira2020hausdorff}, such results are scarce in the literature. 
\end{nb}    

\section{On Full Well-Posedness}\label{sec:stability}
Having established the discoverability of chaotic systems, we now address the second critical component of well-posedness: stability. Specifically, we investigate the behavior of the reconstruction when we do not observe the ground-truth trajectory, but rather a noisy counterpart $y(t) = x(t) + \varepsilon(t)$. To analyze this mathematically, we must specify an estimation procedure; throughout this section, our focus will be on ordinary least squares (OLS).

\subsection{Finite Dimensional Stability}
For finite-dimensional function spaces, it turns out that stability and convergence under noise can be established without any explicit regularization. In such settings, the Gram matrix is strictly invertible, meaning the problem is naturally well-posed. While analogous stability and convergence results can be derived for Hilbert spaces, the infinite-dimensional case fundamentally requires regularization. The natural well-posedness in finite dimensions hinges directly on our prior geometric results via sets of uniqueness, formalized as follows.
\begin{lemma}[Well-Posedness of the Gram Matrix for Sets of Uniqueness]\label{lem:gram}
Consider $\Omega \subseteq \R^d$ state space, let $\mu$ be a probability measure with support $\Lambda = \mathrm{supp}(\mu)\subseteq\Omega$ and let $\mathcal{H}_p$ be a $p$-dimensional vector space of continuous functions on $\Omega$ with basis $\Phi(x) = [\phi_1(x), \dots, \phi_p(x)]^\top$. Either assuming $\phi_i(x)$ are square integrable, or $\Omega$ compact, define the Gram matrix $G_\mu \in \R^{p \times p}$ as:
\begin{equation}
    G_\mu = \int_\Lambda \Phi(x)\Phi(x)^\top \, d\mu(x).
\end{equation}
Then, $\Lambda$ is a set of uniqueness for $\mathcal{H}_p$ if and only if $G_\mu$ is strictly positive definite ($\sigma_{\min}(G_\mu) > 0$).
\end{lemma}
\begin{proof}
Firstly, the assumption that either the basis functions are square-integrable ($\phi_i \in L^2(\mu)$) or $\Omega$ is compact guarantees that the entries of $G_\mu$ are finite, ensuring the matrix is well-defined. By construction, $G_\mu$ is symmetric. For any vector $v \in \R^p$, the associated quadratic form is:
\begin{equation}
    v^\top G_\mu v = v^\top \left( \int_\Lambda \Phi(x)\Phi(x)^\top \, d\mu(x) \right) v = \int_\Lambda (v^\top \Phi(x))^2 \, d\mu(x).
\end{equation}
Since $(v^\top \Phi(x))^2 \geq 0$, we know $v^\top G_\mu v \geq 0$. Thus, $G_\mu$ is always positive semi-definite. Let $P_v(x) = v^\top \Phi(x) \in \mathcal{H}_p$.

($\implies$) Assume $\Lambda$ is a set of uniqueness for $\mathcal{H}_p$ and suppose $v^\top G_\mu v = 0$. Then 
\[\int_\Lambda P_v(x)^2 \, d\mu(x) = 0.\]
As $P_v(x)$ is continuous and $\Lambda = \mathrm{supp}(\mu)$, this integral evaluating to zero strictly implies $P_v(x) = 0$ for all $x \in \Lambda$. Since $\Lambda$ is a set of uniqueness, $P_v|_\Lambda \equiv 0$ implies $P_v \equiv 0$ on $\Omega$, and as $\Phi(x)$ is a basis, meaning entries are linearly independent, $P_v \equiv 0$ requires the coefficient vector to be zero ($v = 0$). Therefore, $v^\top G_\mu v > 0$ for all $v \neq 0$, meaning $G_\mu$ is strictly positive definite.

($\impliedby$) We proceed by proving the contrapositive: assume $\Lambda$ is not a set of uniqueness. Then there exists a non-trivial function $P_v \in \mathcal{H}_p\backslash\{0\}$ such that $P_v(x) = 0$ for all $x \in \Lambda$, but $v \neq 0$. Evaluating the quadratic form for this $v$ yields:
\begin{equation}
    v^\top G_\mu v = \int_\Lambda (P_v(x))^2 \, d\mu(x) = \int_\Lambda 0 \, d\mu(x) = 0.
\end{equation}
Since there exists a non-zero vector $v$ such that $v^\top G_\mu v = 0$, the matrix $G_\mu$ is not strictly positive definite.
\end{proof}

Related full-rank results for polynomial feature matrices under Lorenz-type mixing were established in \citet{tran2017exact}; \Cref{lem:gram} above gives the corresponding criterion for arbitrary finite-dimensional function spaces in terms of sets of uniqueness. While \Cref{lem:gram} ensures the theoretical invertibility of the exact Gram matrix, proving finite-time stability bounds requires us to bound the convergence rate of the empirical integrals. Because Birkhoff's ergodic theorem provides only an asymptotic limit, it is insufficient for finite-time analysis. Therefore, we must constrain our scope from general chaotic systems to those with a quantifiable rate of mixing. Specifically, we restrict our attention to systems admitting an SRB measure \citep{young2002srb} with exponential decay of correlations: a property exhibited, for example, by Anosov flows \citep{sinai2020markov, bowen1975ergodic, bowen1974some}.

\begin{assumption}[Exponential Mixing]\label{ass:mixing}
Assume our domain $X \subseteq \mathbb{R}^n$ is compact, and the true system $\dot{u}=F(u)$ admits 
an ergodic, invariant SRB measure $\mu$ supported on a chaotic attractor $\Lambda$. We will assume that the system is exponentially mixing, i.e. there exist universal system constants $K_0 > 0$ and $\gamma > 0$ such that for any Lipschitz continuous maps $A, B : \Lambda \to \mathbb{R}^{p\times p}$, the auto-covariance satisfies:   
\begin{equation}
\left\| \int_\Lambda A(x)B(\phi^\tau(x))d\mu(x) - \int_\Lambda A(x)d\mu(x) \int_\Lambda B(x)d\mu(x) \right\|_2 \leq K_0 \|A\|_{\mathrm{Lip}} \|B\|_{\mathrm{Lip}} e^{-\gamma |\tau|}.
\end{equation}
\end{assumption}
It is worth noting, that stating mixing for matrix-valued observables in spectral norm lets $K_0$ absorb dimension-dependent factors, which interacts with the explicit $p$ in the final bound. In order to use this to establish explicit stability bounds, we require the following two lemmas:

\begin{lemma}[Finite-Time Matrix Concentration via Covariance Mixing]
\label{lem:markov_concentration}
Let $\Psi(x)$ be an $L^2(\mu)$-orthonormal feature basis and define the zero-mean $p \times p$ matrix process $Z_G(t) = \Psi(x(t))\Psi(x(t))^\top - I_p$. If the system satisfies the exponential mixing condition (\Cref{ass:mixing}), then for any observation time $T > 0$ and failure probability $\eta \in (0, 1)$, the empirical average $\bar{Z}_T = \frac{1}{T} \int_0^T Z_G(t) dt$ satisfies the following with probability at least $1 - \eta$:
\begin{equation}
    \|\bar{Z}_T\|_2 \le \sqrt{ \frac{2 p K_0 \|Z_G\|_{\mathrm{Lip}}^2}{\gamma \eta T} }.
\end{equation}
\end{lemma}
\begin{proof}
We aim to bound the spectral norm $\|\bar{Z}_T\|_2$. Because the Frobenius norm upper bounds the spectral norm ($\|A\|_2 \le \|A\|_F$ for any matrix $A$), it suffices to bound the tail probability of the Frobenius norm. We consider the squared Frobenius norm to apply Markov's inequality:
\begin{equation}
    \mathbb{P}\left(\|\bar{Z}_T\|_2 \ge u\right) \le \mathbb{P}\left(\|\bar{Z}_T\|_F \ge u\right) = \mathbb{P}\left(\|\bar{Z}_T\|_F^2 \ge u^2\right).
\end{equation}

By the definition of the Frobenius norm, we can express the expected squared norm of the integral as a double integral of the trace of the auto-covariance matrix. Using the linearity of the trace and expectation:
\begin{align}
    \mathbb{E}_\mu\left[\|\bar{Z}_T\|_F^2\right] &= \mathbb{E}_\mu\left[ \mathrm{Tr}\left( \bar{Z}_T \bar{Z}_T^\top \right) \right] \nonumber \\
    &= \frac{1}{T^2} \int_0^T \int_0^T \mathrm{Tr} \left( \mathbb{E}_\mu\left[Z_G(t)Z_G(s)^\top\right] \right) dt \, ds.
\end{align}

For any $p \times p$ matrix $A$, the trace is bounded by its spectral norm via $|\mathrm{Tr}(A)| \le p \|A\|_2$. Applying this to the integrand and substituting the exponential mixing bound from \Cref{ass:mixing}:
\begin{equation}
    \left| \mathrm{Tr} \left( \mathbb{E}_\mu\left[Z_G(t)Z_G(s)^\top\right] \right) \right| \le p \left\| \mathbb{E}_\mu\left[Z_G(t)Z_G(s)^\top\right] \right\|_2 \le p K_0 \|Z_G\|_{\mathrm{Lip}}^2 e^{-\gamma|t-s|}.
\end{equation}

We substitute this bound back into the double integral. By leveraging the symmetry of the integration domain $[0,T] \times [0,T]$, we can reduce the double integral over time coordinates $t$ and $s$ into a single integral over the absolute time difference $\tau = t - s$:
\begin{align}
    \int_0^T \int_0^T e^{-\gamma|t-s|} \, dt \, ds &= 2 \int_0^T (T - \tau) e^{-\gamma \tau} \, d\tau \nonumber \\
    &\le 2T \int_0^\infty e^{-\gamma \tau} \, d\tau = \frac{2T}{\gamma}.
\end{align}

Consequently, the expected squared Frobenius norm is bounded by:
\begin{equation}
    \mathbb{E}_\mu\left[\|\bar{Z}_T\|_F^2\right] \le \frac{1}{T^2} \left( p K_0 \|Z_G\|_{\mathrm{Lip}}^2 \right) \left( \frac{2T}{\gamma} \right) = \frac{2 p K_0 \|Z_G\|_{\mathrm{Lip}}^2}{\gamma T}.
\end{equation}

We now apply Markov's inequality to the non-negative random variable $\|\bar{Z}_T\|_F^2$:
\begin{equation}
    \mathbb{P}\left(\|\bar{Z}_T\|_F^2 \ge u^2\right) \le \frac{\mathbb{E}_\mu\left[\|\bar{Z}_T\|_F^2\right]}{u^2} \le \frac{2 p K_0 \|Z_G\|_{\mathrm{Lip}}^2}{\gamma T u^2}.
\end{equation}

To establish the high-probability bound, we equate the right-hand side to our target failure probability $\eta \in (0,1)$:
\begin{equation}
    \eta = \frac{2 p K_0 \|Z_G\|_{\mathrm{Lip}}^2}{\gamma T u^2} \implies u^2 = \frac{2 p K_0 \|Z_G\|_{\mathrm{Lip}}^2}{\gamma \eta T}.
\end{equation}

Taking the square root yields the value of $u$ for which the bound holds with probability at least $1 - \eta$:
\begin{equation}
    u = \sqrt{ \frac{2 p K_0 \|Z_G\|_{\mathrm{Lip}}^2}{\gamma \eta T} }.
\end{equation}
This completes the proof.
\end{proof}

\begin{lemma}[Perturbation of the Inverse]
\label{lem:neumann}
For a strictly positive definite matrix $A$, if a symmetric perturbation satisfies $\|\Delta A\|_2 \le \frac{1}{2}\sigma_{\min}(A)$, then $(A + \Delta A)$ is invertible and:
\begin{equation}
    \|(A + \Delta A)^{-1}\|_2 \le \frac{2}{\sigma_{\min}(A)}.
\end{equation}
\end{lemma}

Under these assumptions, it turns out to be possible to show explicit stability bound for the OLS solution, which guarantees $\hat{F}\to F$ $\mu$, a.s. as $T\to\infty$ and noise goes to $0$.

\begin{theorem}[Finite-Time Stability Bound]
Assume the true dynamics $\dot{x} = F(x)$ belong to a $p$-dimensional vector space $\mathcal{H}_p$ of continuously differentiable functions on $\Omega$, generating a chaotic trajectory with invariant measure $\mu$ supported on a set of uniqueness $\Lambda$. Assuming exponential mixing (\Cref{ass:mixing}), and noisy observations $y(t) = x(t) + \varepsilon(t)$ with bounded noise $\|\varepsilon(t)\|_2 \leq \delta$ and $\|\dot{\varepsilon}(t)\|_2 \leq \delta_v$. Define the intrinsic geometric constants of the hypothesis space over $\mu$:
\begin{equation}
    C_\infty = \sup_{h \in \mathcal{H}_p} \frac{\|h\|_{L^\infty(\Omega)}}{\|h\|_{L^2(\mu)}}, \qquad C_{\mathrm{Lip}} = \sup_{h \in \mathcal{H}_p} \frac{\|\nabla h\|_{L^\infty(\Omega)}}{\|h\|_{L^2(\mu)}}.
\end{equation}

Then, for a sufficiently large observation window $T$ and sufficiently small noise $\delta$, the function-space reconstruction error of the ordinary least squares estimator $\hat{F}_T$ satisfies with probability at least $1 - \eta$:
\begin{equation}
    \|\hat{F}_T - F\|_{L^2(\mu)} \le 2 C_\infty \|F\|_{L^2(\mu)} \left( \frac{\delta_v}{\|F\|_{L^2(\mu)}} + C_{\mathrm{Lip}}\left[3\delta + 4\sqrt{\frac{2 p K_0}{\gamma \eta T}}\right] \right).
\end{equation}
\end{theorem}

\begin{proof}
Firstly, the ordinary least squares estimator is invariant to the choice of basis for the linear hypothesis space $\mathcal{H}_p$, thus without loss of generality we can use an $L^2(\mu)$-orthonormal basis $\Psi(x)$, as by \Cref{lem:gram} the smallest eigenvalue of the gram matrix is non-zero. In this basis, the intrinsic equivalence constants are exactly the basis bounds: $C_\infty = \sup_{x \in \Omega} \|\Psi(x)\|_2$ and $C_{\mathrm{Lip}}$ acting as the Lipschitz constant for $\Psi(x)$. With this, we define the objects we will require: 
\begin{enumerate}[leftmargin=*]
    \item The population gram and cross-covariance. By orthonormality, $G_\mu$ is the identity matrix $I_p$, and the true coefficients $w^*$ satisfy $\|w^*\|_F = \|F\|_{L^2(\mu)}$:
    $$G_\mu=\int_{\Lambda} \Psi(x) \Psi(x)^{\top} d \mu(x) = I_p,\qquad B_\mu = \int_\Lambda \Psi(x) \left( \Psi(x)^\top w^* \right) d\mu(x) = w^*.$$
    \item The noiseless empirical gram and cross-covariance:
    $$
\hat{G}_T=\frac{1}{T} \int_0^T \Psi(x(t)) \Psi(x(t))^{\top} d t, \qquad \hat{B}_T=\frac{1}{T} \int_0^T \Psi(x(t)) \dot{x}(t)^{\top} d t = \hat{G}_T w^*.
$$
    \item The noisy empirical gram and cross-covariance:
    $$
\tilde{G}_T=\frac{1}{T} \int_0^T \Psi(y(t)) \Psi(y(t))^{\top} d t, \qquad \tilde{B}_T=\frac{1}{T} \int_0^T \Psi(y(t)) \dot{y}(t)^{\top} d t = \tilde{G}_T\hat{w}_T.
$$
\end{enumerate}

Secondly, we translate the bounds on the noise into errors on the empirical quantities.
Using the compactness of $\Omega$ and continuous differentiability of $\Psi$, we apply the intrinsic bounds $C_\infty$ and $C_{\mathrm{Lip}}$:
\begin{align*}
    \|\Psi(y)\Psi(y)^\top - \Psi(x)\Psi(x)^\top\|_2 &\leq \|\Psi(y)(\Psi(y) - \Psi(x))^\top\|_2 + \|(\Psi(y) - \Psi(x))\Psi(x)^\top\|_2 \\
    &\leq \|\Psi(y)\|_2 \|\Psi(y) - \Psi(x)\|_2 + \|\Psi(y) - \Psi(x)\|_2 \|\Psi(x)\|_2 \\
    &\leq C_\infty (C_{\mathrm{Lip}} \delta) + (C_{\mathrm{Lip}} \delta) C_\infty = 2 C_\infty C_{\mathrm{Lip}} \delta.
\end{align*}
Integrating over time $T$ preserves this pointwise bound, yielding the explicit deterministic bound for the Gram matrix:
\begin{equation}
    \|\tilde{G}_T - \hat{G}_T\|_2 \leq 2 C_\infty C_{\mathrm{Lip}} \delta \equiv \delta_G.
\end{equation}
Similarly for the cross-covariance:
\begin{align*}
    \|\Psi(y)\dot{y}^\top - \Psi(x)\dot{x}^\top\|_F &\leq \|\Psi(y)\|_2 \|\dot{y} - \dot{x}\|_2 + \|\Psi(y) - \Psi(x)\|_2 \|\dot{x}\|_2 \\
    &\leq C_\infty \delta_v + C_\infty C_{\mathrm{Lip}} \|F\|_{L^2(\mu)} \delta.
\end{align*}
Integrating this yields the explicit deterministic bound for the cross-covariance matrix:
\begin{equation}
    \|\tilde{B}_T - \hat{B}_T\|_F \leq C_\infty \delta_v + C_\infty C_{\mathrm{Lip}} \|F\|_{L^2(\mu)} \delta \equiv \delta_B.
\end{equation}

We next bound the statistical sampling error. The observable is the zero-mean matrix process $Z_G(x) = \Psi(x)\Psi(x)^\top - I_p$. Using the continuous differentiability of $\Psi$ and the intrinsic geometric bounds, we derive its exact Lipschitz constant:
\begin{align*}
    \|Z_G(x) - Z_G(y)\|_2 &= \|\Psi(x)\Psi(x)^\top - \Psi(y)\Psi(y)^\top\|_2 \\
    &\leq \|\Psi(x)(\Psi(x) - \Psi(y))^\top\|_2 + \|(\Psi(x) - \Psi(y))\Psi(y)^\top\|_2 \\
    &\leq 2 C_\infty C_{\mathrm{Lip}} \|x - y\|_2.
\end{align*}
Thus, $\|Z_G\|_{\mathrm{Lip}} \le 2 C_\infty C_{\mathrm{Lip}}$. Applying \Cref{lem:markov_concentration}, with probability at least $1-\eta$, the noiseless empirical Gram matrix satisfies:
\begin{equation}
    \left\|\hat{G}_T - I_p\right\|_2 \leq 2 C_\infty C_{\mathrm{Lip}} \sqrt{\frac{2 p K_0}{\gamma \eta T}} \equiv E_{\text{stat}}.
\end{equation}
As $\hat{B}_T = \hat{G}_T w^*$ and $B_\mu = w^*$, the statistical error of the cross-covariance is bounded directly by $E_{\text{stat}}$:
\begin{equation}
    \left\|\hat{B}_T-B_\mu\right\|_F = \left\|(\hat{G}_T - I_p)w^*\right\|_F \leq \left\|\hat{G}_T - I_p\right\|_2 \|w^*\|_F \leq E_{\text{stat}} \|F\|_{L^2(\mu)}.
\end{equation}

By the triangle inequality, we decompose the empirical matrices into noise perturbations and statistical sampling errors:
\begin{alignat*}{2}
\left\|\tilde{G}_T - I_p\right\|_2 &\leq \underbrace{\left\|\tilde{G}_T-\hat{G}_T\right\|_2}_{\text{Noise Perturbation } (\leq \delta_G)} &&{}+ \underbrace{\left\|\hat{G}_T - I_p\right\|_2}_{\text{Sampling Error \Cref{lem:markov_concentration}} (\leq E_{\mathrm{stat}})} \\
\left\|\tilde{B}_T-B_\mu\right\|_F &\leq \underbrace{\left\|\tilde{B}_T-\hat{B}_T\right\|_F}_{\text{Noise Perturbation } (\leq \delta_B)} &&{}+ \underbrace{\left\|\hat{B}_T-B_\mu\right\|_F}_{\text{Sampling Error} (\leq \|F\|_{L^2(\mu)} E_{\mathrm{stat}})}
\end{alignat*}

Using \Cref{lem:neumann}, for sufficiently small noise and large time, $\delta_G + E_{\mathrm{stat}} \le \frac{1}{2}\sigma_{\min}(I_p) = \frac{1}{2}$, guarantees $\tilde{G}_T$ is strictly invertible with $\|\tilde{G}_T^{-1}\|_2 \le 2$. This gives us the required bound on the OLS parameter error: 
$$\Delta w = \hat{w}_T - w^* = \tilde{G}_T^{-1}(\tilde{B}_T - \tilde{G}_T w^*) = \tilde{G}_T^{-1} \Big( (\tilde{B}_T - B_\mu) - (\tilde{G}_T - I_p)w^* \Big),$$ 
which collecting the bounds above results in:
\begin{align}
    \|\Delta w\|_F &\leq 2 \Big(\delta_B + \delta_G\|F\|_{L^2(\mu)} + 2 E_{\text{stat}}\|F\|_{L^2(\mu)} \Big) \\
    &\leq 2 \left( C_\infty \delta_v + C_\infty C_{\mathrm{Lip}} \|F\|_{L^2(\mu)} \delta + 2 C_\infty C_{\mathrm{Lip}} \|F\|_{L^2(\mu)} \delta + 4 C_\infty C_{\mathrm{Lip}} \|F\|_{L^2(\mu)} \sqrt{\frac{2 p K_0}{\gamma \eta T}} \right) \\
    &= 2 C_\infty \|F\|_{L^2(\mu)} \left( \frac{\delta_v}{\|F\|_{L^2(\mu)}} + C_{\mathrm{Lip}}\left[3\delta + 4\sqrt{\frac{2 p K_0}{\gamma \eta T}}\right] \right).
\end{align}

Finally, because the basis is $L^2(\mu)$-orthonormal, the function space error of the reconstructed vector field simplifies exactly to the parameter error:
\begin{equation}
    \|\hat{F}_T - F\|_{L^2(\mu)}^2 = \int_\Lambda \|\Delta w^\top \Psi(x)\|_2^2 \, d\mu(x) = \mathrm{Tr}( \Delta w^\top I_p \Delta w ) = \|\Delta w\|_F^2.
\end{equation}
Taking the square root gives $\|\hat{F}_T - F\|_{L^2(\mu)} = \|\Delta w\|_F$. Substituting the bound above completes the proof.
\end{proof}
\subsection{Infinite Dimensional Stability}\label{app:infinite_stability}
For the infinite dimensional case, the assumptions themselves must change due to the dimension implicitly appearing in the definition. For this section, we will consider the following infinite dimensional setup: 
\begin{assumption}[Hilbert-Schmidt Exponential Mixing]\label{ass:mixing_hs}
Assume our domain $\Omega \subseteq \mathbb{R}^n$ is compact, and the true system $\dot{x}=F(x)$ admits an ergodic, invariant SRB measure $\mu$ supported on a chaotic attractor $\Lambda$. Let $\mathcal{H}$ be a separable Hilbert space and let $\mathcal{B}_{\mathrm{HS}}(\mathcal{H})$ denote the space of Hilbert-Schmidt operators on $\mathcal{H}$ equipped with the norm $\|\cdot\|_{\mathrm{HS}}$. 

We assume the system is exponentially mixing with respect to this operator topology; i.e., there exist universal system constants $K_0 > 0$ and $\gamma > 0$ such that for any Lipschitz continuous maps $A, B : \Lambda \to \mathcal{B}_{\mathrm{HS}}(\mathcal{H})$, the auto-covariance satisfies:   
\begin{equation}
    \left\| \int_\Lambda A(x)B(\phi^\tau(x)) \, d\mu(x) - \int_\Lambda A(x)\,d\mu(x) \int_\Lambda B(x)\,d\mu(x) \right\|_{\mathrm{HS}} \leq K_0 \|A\|_{\mathrm{Lip}, \mathrm{HS}} \|B\|_{\mathrm{Lip}, \mathrm{HS}} e^{-\gamma |\tau|},
\end{equation}
where $\| \cdot \|_{\mathrm{Lip}, \mathrm{HS}}$ denotes the Lipschitz constant of the map evaluated in the Hilbert-Schmidt norm.
\end{assumption}
\begin{lemma}[Operator Concentration via Hilbert-Schmidt Mixing]
\label{lem:operator_concentration}
Let $\mathcal{H}$ be a separable Hilbert space and $\Psi: \Lambda \to \mathcal{H}$ a bounded feature map. Define the zero-mean covariance anomaly operator $Z_G(t) = \Psi(x(t))\otimes\Psi(x(t)) - G_\mu$. If the system satisfies the Hilbert-Schmidt exponential mixing condition, then for any observation time $T > 0$ and failure probability $\eta \in (0, 1)$, the empirical average $\bar{Z}_T = \frac{1}{T} \int_0^T Z_G(t) dt$ satisfies the following operator norm bound with probability at least $1 - \eta$:
\begin{equation}
    \|\bar{Z}_T\|_{\mathrm{op}} \le \|\bar{Z}_T\|_{\mathrm{HS}} \le \sqrt{ \frac{2 K_0 \|Z_G\|_{\mathrm{Lip}, \mathrm{HS}}^2}{\gamma \eta T} }.
\end{equation}
\end{lemma}

\begin{proof}
We bound the tail probability of the operator norm $\|\bar{Z}_T\|_{\mathrm{op}}$. Because the Hilbert-Schmidt norm strictly upper bounds the operator norm ($\|A\|_{\mathrm{op}} \le \|A\|_{\mathrm{HS}}$ for any operator $A$), we apply Markov's inequality to the squared Hilbert-Schmidt norm:
\begin{equation}
    \mathbb{P}\left(\|\bar{Z}_T\|_{\mathrm{op}} \ge u\right) \le \mathbb{P}\left(\|\bar{Z}_T\|_{\mathrm{HS}}^2 \ge u^2\right).
\end{equation}

By expanding the squared Hilbert-Schmidt norm using its inner product, we can express the expectation as a double integral of the auto-covariance:
\begin{align}
    \mathbb{E}_\mu\left[\|\bar{Z}_T\|_{\mathrm{HS}}^2\right] &= \mathbb{E}_\mu\left[ \left\langle \frac{1}{T} \int_0^T Z_G(t) dt, \frac{1}{T} \int_0^T Z_G(s) ds \right\rangle_{\mathrm{HS}} \right] \nonumber \\
    &= \frac{1}{T^2} \int_0^T \int_0^T \mathbb{E}_\mu\left[ \langle Z_G(t), Z_G(s) \rangle_{\mathrm{HS}} \right] dt \, ds.
\end{align}

We apply the Hilbert-Schmidt exponential mixing assumption to the integrand:
\begin{equation}
    \left| \mathbb{E}_\mu\left[ \langle Z_G(t), Z_G(s) \rangle_{\mathrm{HS}} \right] \right| \le K_0 \|Z_G\|_{\mathrm{Lip}, \mathrm{HS}}^2 e^{-\gamma|t-s|}.
\end{equation}

Substituting this into the double integral and leveraging the symmetry of the domain $[0,T] \times [0,T]$ over the time difference $\tau = t - s$:
\begin{align}
    \int_0^T \int_0^T e^{-\gamma|t-s|} \, dt \, ds &= 2 \int_0^T (T - \tau) e^{-\gamma \tau} \, d\tau \nonumber \\
    &\le 2T \int_0^\infty e^{-\gamma \tau} \, d\tau = \frac{2T}{\gamma}.
\end{align}

Thus, the expected squared Hilbert-Schmidt norm is bounded by:
\begin{equation}
    \mathbb{E}_\mu\left[\|\bar{Z}_T\|_{\mathrm{HS}}^2\right] \le \frac{2 K_0 \|Z_G\|_{\mathrm{Lip}, \mathrm{HS}}^2}{\gamma T}.
\end{equation}

Applying Markov's inequality yields:
\begin{equation}
    \mathbb{P}\left(\|\bar{Z}_T\|_{\mathrm{HS}}^2 \ge u^2\right) \le \frac{2 K_0 \|Z_G\|_{\mathrm{Lip}, \mathrm{HS}}^2}{\gamma T u^2}.
\end{equation}

Setting the right-hand side to the failure probability $\eta$ and solving for $u$ gives the required deviation bound:
\begin{equation}
    u = \sqrt{ \frac{2 K_0 \|Z_G\|_{\mathrm{Lip}, \mathrm{HS}}^2}{\gamma \eta T} }.
\end{equation}
This completes the proof.
\end{proof}

Unfortunately, \Cref{lem:gram} in the infinite dimensional case turns into an injectivity statement, useless for stability. Because of this, regularization becomes necessary. Using the established lemma, we can now prove two results: stability and convergent regularization. 
\begin{theorem}[Finite-Time Stability and Convergent Regularization]
\label{thm:stability_and_convergence}
Assume the true dynamics $\dot{x} = F(x)$ are generated by a parameter $w^* \in \mathcal{H}$ such that $F(x) = \langle w^*, \Psi(x) \rangle_{\mathcal{H}}$, where the feature map $\Psi$ into the separable Hilbert space $\mathcal{H}$ is bounded by $\kappa$ and is Lipschitz continuous with constant $L_\Psi$. Assume \Cref{ass:mixing_hs} holds and that noisy observations are bounded by $\|\varepsilon(t)\|_2 \le \delta$ and $\|\dot{\varepsilon}(t)\|_2 \le \delta_v$. To control the approximation bias, assume the Hölder source condition $w^* = G_\mu^r v$ for some $v \in \mathcal{H}$ and $r \in (0, 1]$. 

Define the \textit{effective observation error} governing the estimator's variance:
\begin{equation}
    \mathcal{E}(\delta, \delta_v, T, \eta) = \frac{\delta_v}{\|w^*\|_{\mathcal{H}}} + L_\Psi \left[ 3\delta + 4\sqrt{\frac{2 K_0}{\gamma \eta T}} \right].
\end{equation}

\textbf{Stability:} For any failure probability $\eta \in (0,1)$ and regularization parameter $\lambda > 0$ satisfying the invertibility threshold $4 \kappa L_\Psi \big( \delta + \sqrt{2 K_0 / \gamma \eta T} \big) \le \lambda$, the regularized estimator $\hat{F}_{T, \lambda}$ satisfies with probability at least $1 - \eta$:
\begin{equation}
    \|\hat{F}_{T, \lambda} - F\|_{L^2(\mu)} \le \kappa \|v\|_{\mathcal{H}} \lambda^r + \frac{2\kappa^2 \|w^*\|_{\mathcal{H}}}{\lambda} \left(\frac{\delta_v}{\|w^*\|_{\mathcal{H}}} + L_\Psi \left[ 3\delta + 4\sqrt{\frac{2 K_0}{\gamma \eta T}} \right]\right).
\end{equation}

\textbf{Convergent Regularization:} By selecting the regularization as $\lambda^* \asymp \mathcal{E}^{\frac{1}{r+1}}$, for any noise level and observation time satisfying $\mathcal{E} \le (2\kappa)^{-\frac{r+1}{r}}$ (particularly as $\delta,\delta_v\to0, T\to\infty$), the invertibility threshold is strictly satisfied. The parameter and function-space errors rigorously converge at the minimax optimal rate $\mathcal{O}(\mathcal{E}^{\frac{r}{r+1}})$:
\begin{equation}
    \|\hat{w}_{T, \lambda^*} - w^*\|_{\mathcal{H}} \le \frac{1}{\kappa} \|\hat{F}_{T, \lambda^*} - F\|_{L^2(\mu)} \le \Big( \|v\|_{\mathcal{H}} + 2\kappa \|w^*\|_{\mathcal{H}} \Big) \, \mathcal{E}^{\frac{r}{r+1}}.
\end{equation}
\end{theorem}

\begin{proof}
We analyze the parameter error by decomposing it into approximation error (bias) and estimation error (variance):
\begin{equation}
    \|\hat{w}_{T, \lambda} - w^*\|_{\mathcal{H}} \le \underbrace{\|\hat{w}_{T, \lambda} - w^*_\lambda\|_{\mathcal{H}}}_{\text{Estimation Error}} + \underbrace{\|w^*_\lambda - w^*\|_{\mathcal{H}}}_{\text{Approximation Error}}.
\end{equation}

We first bound the approximation error. By definition, $w^*_\lambda = (G_\mu + \lambda I)^{-1} G_\mu w^*$. Using the source condition $w^* = G_\mu^r v$ and standard spectral calculus for self-adjoint operators (noting $\sup_{s \ge 0} \frac{\lambda s^r}{s + \lambda} \le \lambda^r$ for $r \in (0, 1]$), the bias is bounded by:
\begin{equation}
    \|w^*_\lambda - w^*\|_{\mathcal{H}} = \left\| -\lambda (G_\mu + \lambda I)^{-1} G_\mu^r v \right\|_{\mathcal{H}} \le \lambda^r \|v\|_{\mathcal{H}}.
\end{equation}

We then bound the estimation error. Algebraically isolating the perturbations yields: 
\begin{align*}
    \hat{w}_{T, \lambda} - w^*_\lambda &= (\tilde{G}_T + \lambda I)^{-1} \Big( \tilde{B}_T - (\tilde{G}_T + \lambda I) w^*_\lambda \Big) \\
    &= (\tilde{G}_T + \lambda I)^{-1} \Big( (\tilde{B}_T - B_\mu) - (\tilde{G}_T - G_\mu) w^*_\lambda \Big).
\end{align*}
Taking the Hilbert space norm, and using $\|w^*_\lambda\|_{\mathcal{H}} \le \|w^*\|_{\mathcal{H}}$:
\begin{equation}
    \|\hat{w}_{T, \lambda} - w^*_\lambda\|_{\mathcal{H}} \le \|(\tilde{G}_T + \lambda I)^{-1}\|_{\mathrm{op}} \Big( \|\tilde{B}_T - B_\mu\|_{\mathcal{H}} + \|\tilde{G}_T - G_\mu\|_{\mathrm{op}} \|w^*\|_{\mathcal{H}} \Big).
\end{equation}

We deterministically bound the noise perturbations using $\kappa$ and $L_\Psi$:
$\|\tilde{G}_T - \hat{G}_T\|_{\mathrm{op}} \le 2 \kappa L_\Psi \delta$ and $\|\tilde{B}_T - \hat{B}_T\|_{\mathcal{H}} \le \kappa \delta_v + \kappa L_\Psi \|w^*\|_{\mathcal{H}} \delta$.
Using \Cref{lem:operator_concentration}, the statistical error is $\|\hat{G}_T - G_\mu\|_{\mathrm{op}} \le 2 \kappa L_\Psi \sqrt{\frac{2 K_0}{\gamma \eta T}} \equiv E_{\mathrm{stat}}$.

The total operator perturbation is bounded by $\Delta_G = 2 \kappa L_\Psi \delta + E_{\mathrm{stat}}$. The invertibility threshold $\Delta_G \le \frac{\lambda}{2}$ ensures via the Neumann series that $\|(\tilde{G}_T + \lambda I)^{-1}\|_{\mathrm{op}} \le \frac{2}{\lambda}$.
Substituting the noise, statistical, and inverse bounds into the estimation error gives:
\begin{align*}
    \|\hat{w}_{T, \lambda} - w^*_\lambda\|_{\mathcal{H}} &\le \frac{2}{\lambda} \Big( \underbrace{\kappa \delta_v + \kappa L_\Psi \|w^*\|_{\mathcal{H}} \delta}_{\|\tilde{B}_T - \hat{B}_T\|_{\mathcal{H}}} + \underbrace{E_{\mathrm{stat}} \|w^*\|_{\mathcal{H}}}_{\|\hat{B}_T - B_\mu\|_{\mathcal{H}}} + \underbrace{\left( 2 \kappa L_\Psi \delta + E_{\mathrm{stat}} \right)}_{\|\tilde{G}_T - G_\mu\|_{\mathrm{op}}} \|w^*\|_{\mathcal{H}} \Big) \\
    &= \frac{2\kappa}{\lambda} \left( \delta_v + 3 L_\Psi \|w^*\|_{\mathcal{H}} \delta + \frac{2}{\kappa} E_{\mathrm{stat}} \|w^*\|_{\mathcal{H}} \right).
\end{align*}
Factoring out $\|w^*\|_{\mathcal{H}}$ and substituting $E_{\mathrm{stat}}$, the bracketed term collapses exactly into the effective error $\mathcal{E}$, simplifying the estimation variance to:
\begin{equation}
    \|\hat{w}_{T, \lambda} - w^*_\lambda\|_{\mathcal{H}} \le \frac{2\kappa \|w^*\|_{\mathcal{H}}}{\lambda} \mathcal{E}.
\end{equation}

The function space error relates to the parameter error via the covariance operator: $\|F - \hat{F}_{T, \lambda}\|_{L^2(\mu)}^2 = \langle \Delta w, G_\mu \Delta w \rangle_{\mathcal{H}} \le \|G_\mu\|_{\mathrm{op}} \|\Delta w\|_{\mathcal{H}}^2$. Since $\|G_\mu\|_{\mathrm{op}} \le \kappa^2$, we have $\|F - \hat{F}_{T, \lambda}\|_{L^2(\mu)} \le \kappa \|\Delta w\|_{\mathcal{H}}$.
Multiplying our parameter bias and variance bounds by $\kappa$ completes the proof of general stability.

To prove minimax convergent regularization, we explicitly substitute the regularization parameter $\lambda^* = \mathcal{E}^{\frac{1}{r+1}}$ into the parameter error bound. Evaluating the approximation bias term yields:
\begin{equation}
    \|v\|_{\mathcal{H}} (\lambda^*)^r = \|v\|_{\mathcal{H}} \left( \mathcal{E}^{\frac{1}{r+1}} \right)^r = \|v\|_{\mathcal{H}} \mathcal{E}^{\frac{r}{r+1}}.
\end{equation}
Evaluating the estimation variance term yields:
\begin{equation}
    \frac{2\kappa \|w^*\|_{\mathcal{H}}}{\lambda^*} \mathcal{E} = 2\kappa \|w^*\|_{\mathcal{H}} \, \mathcal{E} \left( \mathcal{E}^{\frac{1}{r+1}} \right)^{-1} = 2\kappa \|w^*\|_{\mathcal{H}} \mathcal{E}^{\frac{r}{r+1}}.
\end{equation}
Summing these evaluated components and factoring out the common geometric rate $\mathcal{E}^{\frac{r}{r+1}}$ provides the optimal parameter convergence bound, from which the function space bound directly follows via the $\kappa$ scaling. This matches the known minimax optimal learning rate for inverse problems under a Hölder source condition of regularity $r$ \citep{caponnetto2007optimal}.

Finally, we rigorously verify that this optimal choice of $\lambda^*$ satisfies the assumed invertibility threshold. We can explicitly bound the left-hand side of the threshold:
\begin{align}
    4 \kappa L_\Psi \left( \delta + \sqrt{\frac{2 K_0}{\gamma \eta T}} \right) &\le 2\kappa \left( \frac{\delta_v}{\|w^*\|_{\mathcal{H}}} + L_\Psi \left[ 3\delta + 4\sqrt{\frac{2 K_0}{\gamma \eta T}} \right] \right) = 2\kappa \mathcal{E}.
\end{align}
To ensure the threshold holds, it is therefore sufficient to enforce $2\kappa \mathcal{E} \le \lambda^*$. Substituting $\lambda^* = \mathcal{E}^{\frac{1}{r+1}}$ and dividing by $\mathcal{E}$ (assuming $\mathcal{E} > 0$) yields:
\begin{equation}
    2\kappa \le \mathcal{E}^{\frac{1}{r+1} - 1} = \mathcal{E}^{-\frac{r}{r+1}}.
\end{equation}
Rearranging yields the required condition $\mathcal{E} \le (2\kappa)^{-\frac{r+1}{r}}$. Thus, for a sufficiently small effective observation error $\mathcal{E}$, the invertibility threshold is unconditionally satisfied by the optimal regularization parameter, completing the proof.
\end{proof}

\section{On Conservation Laws}\label{sec:cons_laws}
\begin{figure}[t]
    \centering
    \includegraphics[width=.5\linewidth]{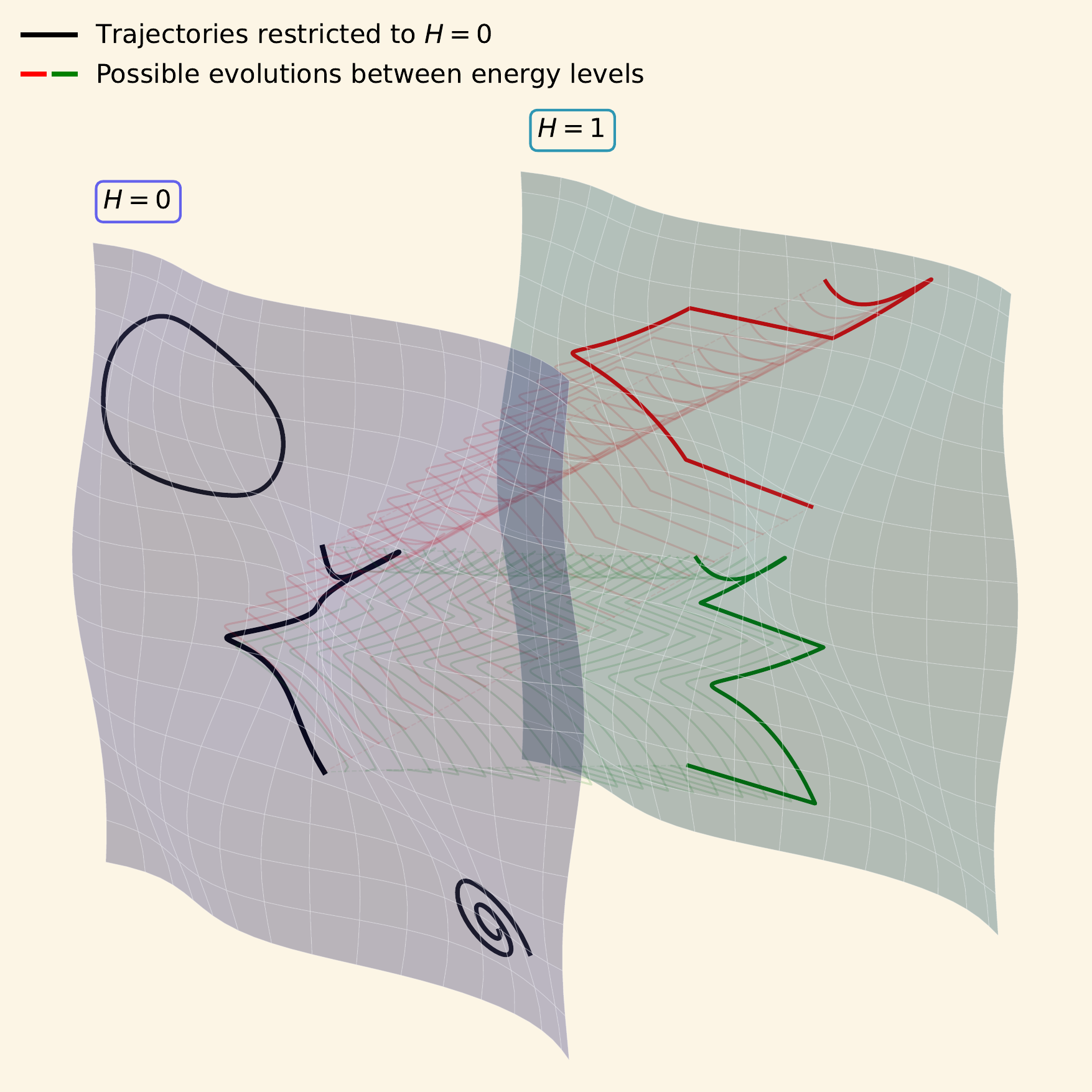}
    \caption{\textbf{Illustration of the main argument behind \Cref{thm:first_integral}.} An analytic first integral foliates the underlying space into analytic surfaces, to which each trajectory is restricted. As each analytic surface is a zero set of an analytic first integral, it is not a set of uniqueness, and trajectories can evolve arbitrarily between the leaves of the foliation, while still preserving the first integral, leading to non-discoverability.}
    \label{fig:firstintegral_indiscovery}
\end{figure}
The main theorems from the previous section highlight that there are significant barriers to discoverability, begging the question: is it impossible to discover non-chaotic systems? 

Luckily, oftentimes we possess information about the system that goes beyond simply the trajectories. We have access to conservation laws, oftentimes represented via equalities $H(\dot{u},u) = 0$. Two particular edge cases of interest that arise are when the laws are completely determining $\big(H:= \sum_i\left(\dot{u}_i-F_i(u)\right)^2\big)$ and completely useless $\big(H:=0\big)$ for determining dynamics of the system. In general, conservation laws are going to live somewhere in between \citep{cory2024evolving,srivastava2025ai}. This raises the following question: when can they provably be useful? In this section, we restrict our attention to the setting of analytic discovery. 

\subsection{First integrals}
We first begin by showing that the existence of first integrals, i.e., conservation laws without dependence on derivatives with $H(u,\dot u) = G(u)$, turns out to be detrimental for discovery, illustrated in \Cref{fig:firstintegral_indiscovery}.

\begin{theorem}\label{thm:first_integral}
If an analytic vector field $F$ admits a non-constant global analytic first integral $G$, then no trajectory of $F$ is a set of uniqueness, and thus $F$ is not analytically discoverable.
\end{theorem}
\begin{proof}
Let $G:\Rl^d\to\Rl$ be a non-constant global analytic first integral for the system $\dot u = F(u)$. I.e. $\nabla G\cdot F = 0$. Our goal is to show that an arbitrary trajectory is not a set of uniqueness. Let $x_0\in\Rl^d$ be any initial point and let $S=\Orb(x_0,F)$ be the trajectory passing through it. To show that $S$ is not a set of uniqueness, we must construct a non-zero analytic function $H$ that vanishes on $S$. The time derivative of $G$ along this trajectory is:
$$ \frac{d}{dt} G(u(t)) = \nabla G(u(t)) \cdot \dot{u}(t) = \nabla G(u(t)) \cdot F(u(t)) = 0 $$
This implies that $G(u(t))$ is constant for all $t$. Let this constant be $c_0 = G(v)$. Define a new function $H(u) = G(u) - c_0$. Then as $G$ is a non-constant analytic function, $H$ is an analytic function which isn't identically zero. Now take $u$ any point on the trajectory $\Orb(x, F)$. Then $u = u(t)$ for some $t$ and hence $H(u) = G(u(t)) - c_0 = 0$. 
Therefore, $S$ is not a set of uniqueness, and by \Cref{prop:general-uniqueness}, $F$ is not analytically discoverable. To go to finitely many, consider trajectories with constants $c_1, \dots, c_n$, the product of the functions $(G - c_i)$ vanishes on the union and is nonzero because $G$ is non-constant. 
\end{proof}
\begin{nb}
Now this may seem contradictory, as this would imply that there can be no chaotic systems with first integrals, to which one could bring up, e.g., Henon-Heiles or Double Pendulum as counterexamples. However, both of these systems are \emph{not chaotic} in the position-momentum variables. Instead, they are only chaotic in certain cross-sections. 
\end{nb}

\subsection{Conservation laws recovering uniqueness}
The theorem in the previous subsection demonstrates that constraints imposed by first integrals effectively confine system trajectories to a lower-dimensional manifold (see \Cref{fig:firstintegral_indiscovery} for a visualization), thereby preventing discoverability due to insufficient state-space exploration. This raises the question of whether incorporating these integrals as prior knowledge could overcome this limitation. Using the results from \Cref{lem:uniq_w_extra} and \Cref{thm:uniq_symmjac}, we show that leveraging this information can make discovery of even non-chaotic systems possible. 
\begin{theorem}%
\label{thm:full_uniq_symmjac}
An analytic system $F$, admitting a conservation law $H$ with $\nabla_{1} H(\dot{u}, u) = 0$ for any $u$ solution, is  discoverable within the space of vector fields with symmetric (skew-symmetric) Jacobian if for some trajectory $\gamma$, there exists an analytic submanifold $\mc{S}$ and a point $p\in\mathcal{S}\subseteq\operatorname{cl}\gamma$, such that $\ker{\nabla_{1}^2 H(F(p), p)} \subseteq \Tp{\mc{S}}$.
\end{theorem}
\begin{proof}
The proof is the direct application of \Cref{thm:uniq_symmjac}. 
\end{proof}
\begin{nb}
    The statement appearing here is less general than \Cref{thm:uniq_symmjac}, however it appears to be most useful in this context. In particular, if we are considering an orbit $\gamma$ covering an attractor, we can pick the manifold to be the integral surface of the non-contracting directions (unstable and flow directions) of the dynamics, resulting in tangent spaces with dimension larger than $1$.
\end{nb}

\begin{corollary}
    The 2D harmonic oscillator is discoverable from any single trajectory away from the origin, given the conservation law 
    $H = \frac12(x\dot{x}+y\dot{y})^2+ \frac12(\dot{x}^2 + \dot{y}^2 - x^2 - y^2)^2$.
\end{corollary}
\begin{proof}
    For this, we simply need to check that the conditions of \Cref{lem:uniq_w_extra} hold. 

Let $u = (x,y)$ and $v = (\dot{x},\dot{y})$. The function can be written as:
$$ H(v, u) = \frac{1}{2}(u \cdot v)^2 + \frac{1}{2}(\|v\|^2 - \|u\|^2)^2.$$
The gradient of $H$ with respect to $v$ is:
$$ \nabla_{v} H = (u \cdot v)u + 2(\|v\|^2 - \|u\|^2)v $$
At the solution $v = F(u)$, we have $u \cdot F(u) = 0$ and $\|F(u)\|^2 = \|u\|^2$, which gives $\nabla_{v} H(F(u), u)= 0$. The Hessian matrix, $\operatorname{hess}_{v}H$, evaluated at the solution simplifies to:
$$ \nabla^2_{v}H(F(u),u) = uu^{\top} + 4F(u)F(u)^{\top} = 
\begin{pmatrix} x^2+4y^2 & -3xy \\ -3xy & y^2+4x^2 \end{pmatrix} $$
The determinant is:
$\det(\nabla^2_{v}H) = (x^2+4y^2)(y^2+4x^2) - (-3xy)^2 = 4(x^2+y^2)^2 = 4\|u\|^4 $. 
Thus, for any $u$ away from the origin, the Hessian is non-singular.
\end{proof}

\begin{nb}
    Note that the classical Hamiltonian or its derivatives themselves would not be satisfactory in this case. Intuitively, this arises because the same periodic trajectory can be achieved via a radius-dependent angular speed. Thus, in order to achieve uniqueness, one requires specifying the speed itself, which the $H$ above does. 
\end{nb}

\begin{corollary}\label{cor:uniq_sho_ap}
    The 2D harmonic oscillator is discoverable within the space of vector fields with skew-symmetric Jacobians from any single trajectory away from the origin, given the conservation law 
    $H = \frac12(x\dot{x}+y\dot{y})^2$
\end{corollary}
\begin{proof}
    For this, we simply need to check that the conditions of \Cref{thm:uniq_symmjac} hold. Let $H(v, u) = (u^\top v)^2$. The gradient with respect to $v$ is $\nabla_{v} H = 2(u^\top v)u = 0$ on the trajectory. The Hessian of $H$ with respect to $v$ is $C = \nabla_{v}^2 H = 2uu^\top$. Let us pick the point $p = (1, 0) \in \mathcal{M}$.
At this point, the Hessian is:
$$C(p) = 2\begin{pmatrix} 1 \\ 0 \end{pmatrix}\begin{pmatrix} 1 & 0 \end{pmatrix} = \begin{pmatrix} 2 & 0 \\ 0 & 0 \end{pmatrix}$$
The kernel of this matrix is $\ker(C) = \text{span}\left\{(0, 1)^\top\right\}$.
The tangent space to the circle $\mathcal{M}$ at $p=(1,0)$ is the vertical line, so $T_{p}\mathcal{M} = \text{span}\left\{(0, 1)^\top\right\}$.
Therefore, $\ker(C) = T_{p}\mathcal{M}$, which satisfies the condition $\ker(C) \subseteq T_{p}\mathcal{M}$.

\end{proof}

The results above constructively show that in order to be able to discover non-chaotic systems, it becomes necessary to have a-priori knowledge in the form of conservation laws.

\section{Computational Details}\label{sec:comp_details}
\begin{figure}[t]
    \centering
    \begin{subfigure}[b]{0.49\linewidth}
        \centering
        \includegraphics[width=\linewidth]{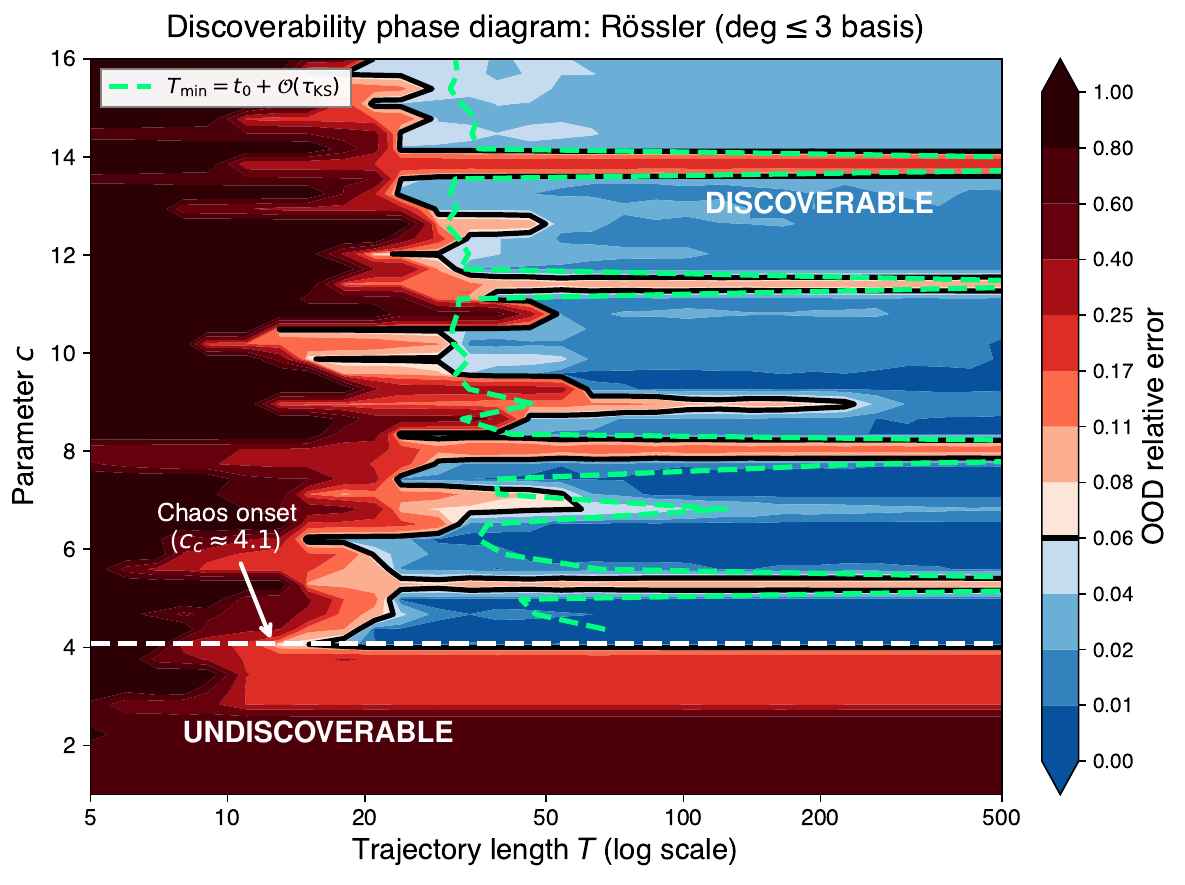}
        \caption{R\"ossler system}
        \label{fig:rossler_ood}
    \end{subfigure}
    \hfill
    \begin{subfigure}[b]{0.49\linewidth}
        \centering
        \includegraphics[width=\linewidth]{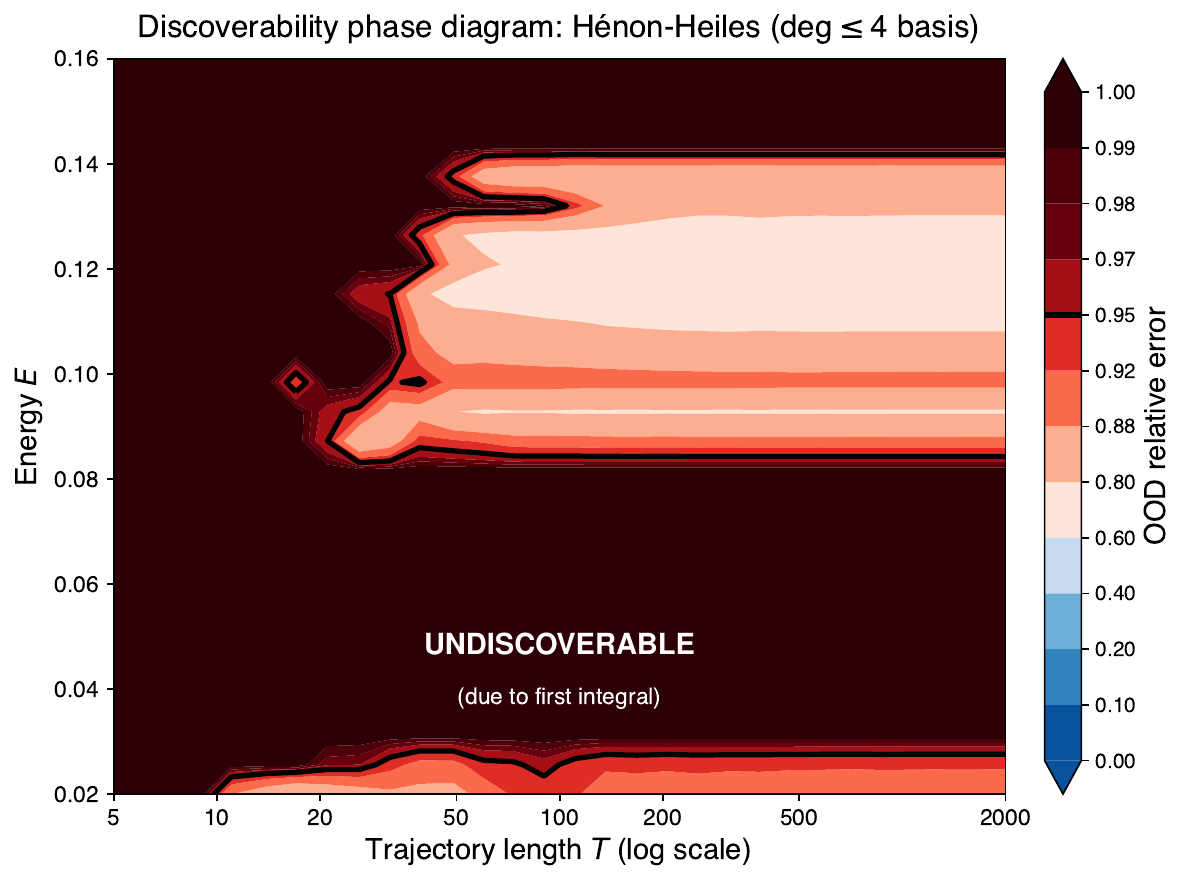}
        \caption{H\'enon-Heiles system}
        \label{fig:henonheiles_ood}
    \end{subfigure}
    \caption{Discoverability phase diagrams for the Rössler and Hénon-Heiles systems. (a) OOD prediction error for the Rössler system as a function of trajectory length $T$ and bifurcation parameter $c$, with chaos onset marked at $c_c$ and the $T_{\min}$ boundary fitted as $T_{\min} = t_0 + \mathcal{O}(\tau_{\mathrm{KS}})$. (b) OOD prediction error for the Hénon-Heiles system as a function of $T$ and energy $E$. The system is structurally undiscoverable across all trajectory lengths due to the conserved first integral $H = E$ by \Cref{thm:first_integral}, which prevents generalization to off-shell test points.}
    \label{fig:rossler_henonheiles_ood}
\end{figure}
\paragraph{Experimental Setup and Computational Details.}
We analyze three systems: (i) Lorenz-96 system; (ii) R\"ossler system; and (iii) H\'enon-Heiles. Across all three systems, the discoverability pipeline is identical. 
For the discoverability phase diagram, we ask: given a trajectory of length $T$, can a model fitted on that trajectory generalize to unseen states? We represent the vector field $\dot{u} = F(u)$ using a \emph{dictionary} of candidate basis functions, and learn the mapping from states to derivatives via ridge regression on the observed trajectory, which we refer to as the \emph{discovered model}: 
$$
\hat{\theta} = \arg\min_{\theta} |\Phi\theta - \dot{u}|^2 + \lambda_r |\theta|^2 
= (G + \lambda_r I)^{-1} \Phi^\top \dot{u},
$$
where $\dot{u} \in \mathbb{R}^{m \times d}$ contains the observed derivatives at each of the $m$ data points, $\Phi \in \mathbb{R}^{m \times p}$ is the design matrix whose rows are the $p$ evaluated basis functions at each data point, and $\theta \in \mathbb{R}^{p \times d}$ contains the corresponding regression coefficients. Here, $d$ is the state-space dimension, $p$ is the number of basis functions, and $G = \Phi^\top \Phi$ is the Gram matrix. The regularization parameter is set adaptively as $\lambda_r = \alpha \cdot \mathrm{tr}(G)/(mp)$, for $\alpha$ a system-specific coefficient.

Trajectories of varying lengths $T$ (30 logarithmically spaced values) are generated and OOD generalization is assessed by evaluating the discovered model on $N_\mathrm{test} = 300$ test points sampled uniformly from a bounding box extending $\pm 0.3$ times the per-dimension range of the training trajectory.

The relative prediction error is computed component-wise, normalized by the mean absolute derivative magnitude, and averaged to yield a scalar OOD error. The system is deemed \emph{discoverable} at trajectory length $T$ if the OOD error falls below the threshold $\epsilon_\mathrm{thr}$. The minimum discoverable trajectory length $T_\mathrm{min}$ is empirically fitted as $T_\mathrm{min} = t_0 + \mathcal{O}(\tau_\mathrm{KS})$, where $\tau_\mathrm{KS} = 1/h_\mathrm{KS}$ is the Kolmogorov time, via least squares on the chaotic forcing values.

The Lyapunov spectrum is computed by integrating the variational equations alongside the trajectory with periodic QR decomposition, using a warmup of $T_{\text{warm}}$ time units (discarded) and a measurement period of $T$ time units. The maximum Lyapunov exponent $\lambda_1$ and KS entropy $h_\mathrm{KS} = \sum_{\lambda_i > 0} \lambda_i$ characterize the chaotic regime, with chaos onset identified via a near-zero threshold on $\lambda_1$. All trajectories are generated with a fourth-order Runge-Kutta integrator, except where noted.

For the model ambiguity analysis, we quantify uncertainty in the discovered model by examining the posterior covariance of the ridge regression coefficients, which is proportional to $\Sigma = G_\mathrm{reg}^{-1} = (G + \lambda_r I)^{-1}$. Rather than scaling by a noise variance estimate, we directly report $\mathrm{Tr}(\Sigma)/(mp)$ as a normalized measure of the average spread of the posterior over parameters, capturing how geometrically constrained the regression problem is at each forcing value. A large $\mathrm{Tr}(\Sigma)/(mp)$ indicates that many directions in parameter space remain poorly determined by the data --- i.e., the model is \emph{ambiguous}, even with the longest available trajectory.

\paragraph{Lorenz-96.}
The system being analyzed is $$\dot{u}_i = (u_{i+1} - u_{i-2})u_{i-1} - u_i + F, \qquad i = 1, \ldots, d \pmod{d},$$
with $d = 40$ and periodic boundary conditions.
The forcing $F \in [3, 8]$ (26 uniformly spaced values), cubic monomial basis ($\leq 3$), $T \in [100, 5000]$ (30 log-spaced), $\Delta t = 0.05$, chaos onset at $F_c \approx 4.2$ ($\lambda_1 > 0.005$), $\epsilon_\mathrm{thr} = 0.10$, $\alpha = 10^{-6}$. 
The initial condition $x_0 = F \cdot {r}$ where ${r} \sim \mathrm{Uniform}([0.9, 1.1])^{40}$ independently per component (seed $42 + f_i$), i.e. each component is initialised near the constant solution $u_i = F$ and a warmup $T_{\text{warm}}=5000$. The minimum discoverable trajectory length is fitted as $T_\mathrm{min} = t_0 + \mathcal{O}(\tau_\mathrm{KS})$ with $t_0 = 431.1$ and scaling $781.5$.

\paragraph{R\"ossler.} The system being analyzed is 
$$\dot{x}_0 = -x_1 - x_2, \qquad \dot{x}_1 = x_0 + a x_1, \qquad \dot{x}_2 = b + x_2(x_0 - c)$$
with fixed $a = 0.2$, $b = 0.2$, and $c$ as the bifurcation parameter.
$d = 3$ variables $(x_0, x_1, x_2)$, fixed $a = 0.2$, $b = 0.2$, sweep $c \in [1, 16]$ (50 uniformly spaced values), cubic monomial basis ($\leq 3$), $T \in [5, 500]$, $\Delta t = 0.005$, chaos onset at $c_c$ ($\lambda_1 > 0.002$), $\epsilon_\mathrm{thr} = 0.06$, $\alpha = 10^{-6}$. The initial condition is fixed at $x_0 = (0, -6, 0.02)$ for all values of $c$, with a warmup $T_{\text{warm}}=500$. The minimum discoverable trajectory length is fitted as $T_\mathrm{min} = t_0 + \mathcal{O}(\tau_\mathrm{KS})$ with $t_0 = 11.0$ and scaling $2.3$.

\paragraph{H\'enon-Heiles.} The system being analyzed is 
$$\dot{q}_1 = p_1, \qquad \dot{q}_2 = p_2, \qquad \dot{p}_1 = -q_1 - 2q_1 q_2, \qquad \dot{p}_2 = -q_2 - q_1^2 + q_2^2$$
$d = 4$ variables $(q_1, q_2, p_1, p_2)$, Hamiltonian $H = \tfrac{1}{2}(p_1^2+p_2^2) + \tfrac{1}{2}(q_1^2+q_2^2) + q_1^2 q_2 - \tfrac{1}{3}q_2^3 = E$, energy swept over $E \in [0.02, 0.16]$ (26 uniformly spaced values, strictly below the escape energy $E_\mathrm{esc} = 1/6$), quartic monomial basis ($\leq 4$), $T \in [5, 2000]$, $\Delta t = 0.005$. A single fixed initial condition is used for each $E$: $q_1 = 0.1$, $q_2 = 0$, $p_1 = 0$, and $p_2$ is solved from $H = E$:
$$p_2 = \sqrt{2E - q_1^2 - q_2^2 - 2q_1^2 q_2 + \tfrac{2}{3}q_2^3} = \sqrt{2E - 0.01}$$
(since $q_2 = 0$ simplifies the potential terms), with a warmup $T_{\text{warm}}=500$. Trajectories are integrated with a Störmer–Verlet symplectic integrator (conserving $H$ to $\mathcal{O}(\Delta t^2)$); Lyapunov variational equations use RK4. $\epsilon_\mathrm{thr} = 0.95$, $\alpha = 10^{-3}$. Because $H = E$ is a conserved first integral, the vector field on the energy hypersurface is not uniquely identifiable from data in ambient $\mathbb{R}^4$: OOD test points drawn from the ambient bounding box lie generically off the energy shell, so the OOD error never falls below $\approx 0.9$ regardless of trajectory length. The system is therefore structurally undiscoverable in this formulation, and no $T_\mathrm{min}$ fit is reported.

\end{document}